\documentclass[12pt]{amsart}
\usepackage{graphicx}
\usepackage{multicol,multirow}
\usepackage{amsmath,amssymb,amsfonts}
\usepackage{mathrsfs}
\usepackage{amsthm}
\usepackage{rotating}
\usepackage{appendix}
\usepackage{ifpdf}
\usepackage{hyperref}
\usepackage[T1]{fontenc}
\usepackage{newtxtext}
\usepackage{newtxmath}
\usepackage{textcomp}
\usepackage{xcolor}
\usepackage{lipsum}

\newtheorem{theorem}{Theorem}[section]
\newtheorem{lemma}[theorem]{Lemma}
\theoremstyle{definition}
\newtheorem{definition}{Definition}
\theoremstyle{plain}
\newtheorem{proposition}{Proposition}
\newtheorem{claim}{Claim}

\theoremstyle{plain}
\newtheorem{corollary}{Corollary}
\theoremstyle{plain}
\newtheorem{conjecture}{Conjecture}
\title{On roundness of rotation sets}

\begin{document}

 %	 e.g. Paper, Letter, Topical Review...

\author{Boris Perrot, Jan Boro\'nski and Alex Clark}

\address[Boris Perrot]{ENS Paris-Saclay (Université Paris-Saclay), 4, avenue des Sciences, 91190 Gif-sur-Yvette, France}
\email{boris.perrot@ens-paris-saclay.fr}

\address[J. Boro\'nski]{Faculty of Mathematics and Computer Science, Jagiellonian University, ul. Łojasiewicza 6, 30-348, Krak\'ow, Poland -- and -- National Supercomputing Centre IT4Innovations\\ 
	IRAFM, University of Ostrava\\
	30. dubna 22, 70103 Ostrava, Czech Republic}
\email{jan.boronski@uj.edu.pl, jan.boronski@osu.cz}

\address[Alex Clark]{School of Mathematical Sciences, Queen Mary University of London, Mile End Road, E1 4NS, London,UK}
\email{alex.clark@qmul.ac.uk}

\maketitle

\begin{abstract}
Motivated by the question whether a round disk can be realized as the rotation set of a torus diffeomorphism, we study the roundness of rotation sets of a parametric family of torus diffeomorphisms $F_\rho$, where the parameter $\rho$ ranges over irrational numbers in $(0,1)$. Each $F_\rho$ is a Kwapisz-like diffeomorphism with a 2-dimensional non-polygonal rotation set 
  $$\Lambda'_\rho = \operatorname{conv}\left(\left\{(\pm\frac{\lceil m\rho \rceil}{m+n+1}, \pm\frac{\lceil n\rho \rceil}{m+n+1}): m, n \in \mathbb{N} _0, \lceil m\rho\rceil - m\rho<\rho,\lceil n\rho\rceil - n\rho<\rho\right\}\right)$$  whose extreme point set contains exactly four (two-sided) accumulation points. We define the roundness of $\Lambda'_\rho$ as the ratio $R_\rho=\frac{\operatorname{Area}(\Lambda'_\rho)}{\pi\rho^2}$, and give its upper and lower bounds in terms of $\rho$. 
 $R_\rho$ is neither monotone nor continuous. 
\end{abstract}
\begin{figure}[!h]
\begin{center}
\includegraphics[width=11cm]{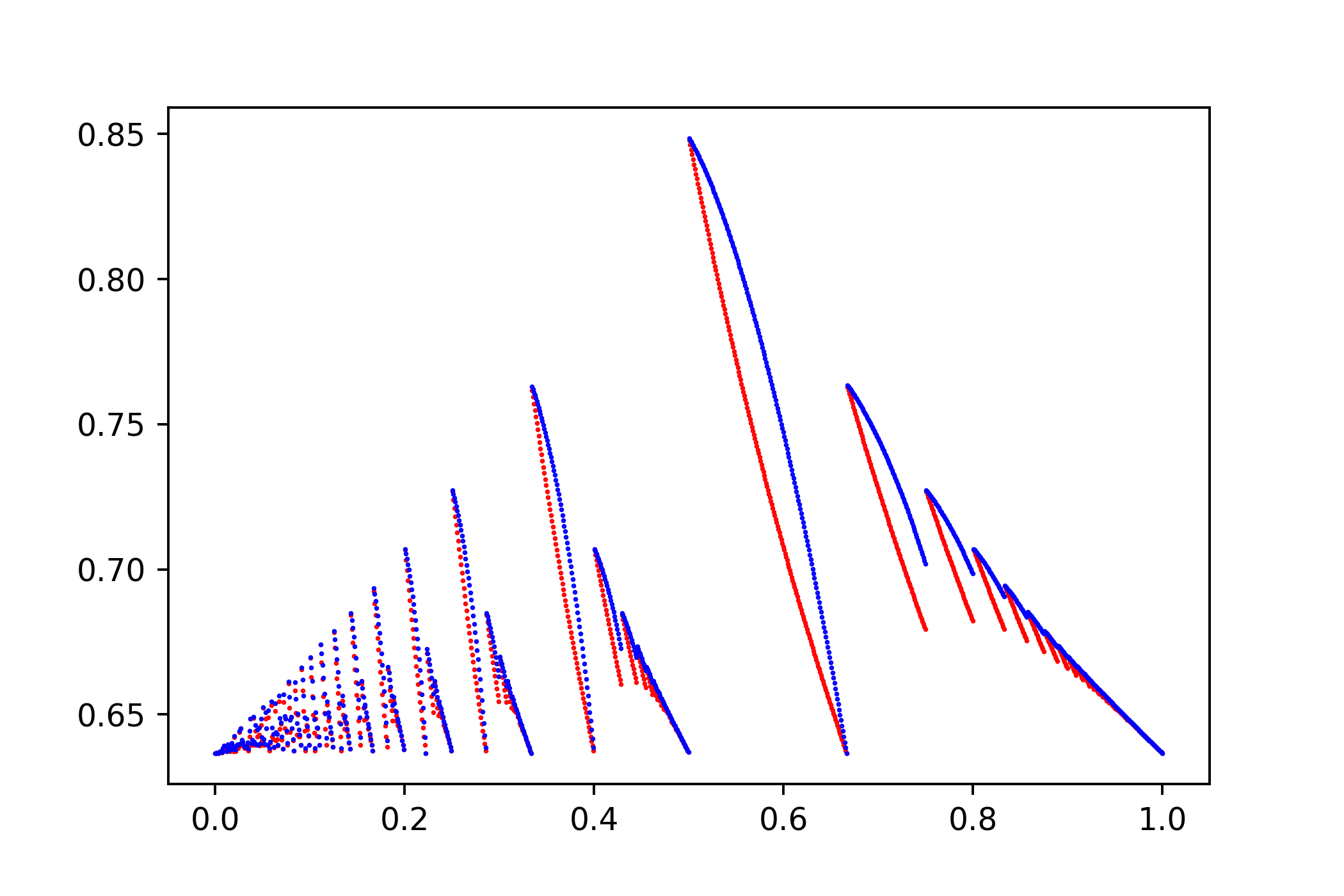}
\end{center}
    \vskip-0.4cm\caption{Lower (red) and upper (blue) bounds for roundness $R_\rho$ of the rotation set $\Lambda'_\rho$ as a function of $\rho$}
		\label{fig:roundness}
\end{figure}
\section{Introduction}
\subsection{Motivation and main results}
This paper is motivated by a well-known open problem in rotation theory, whether the round disk can appear as the rotation set of a torus diffeomorphism; see e.g. \cite{JKT}. Several researchers have worked towards a solution to this problem in the last 35 years, but beyond examples with convex polygons of arbitrary shape \cite{Kwapisz1}, and examples of non-polygonal rotation sets with countably many extreme points and one or two accumulation points (first by Kwapisz \cite{Kwapisz2}, and then by Boyland, de Carvalho, and Hall \cite{BCH2}) no progress has been made for nearly a decade now. For the moment, we propose to pause the search for examples with a richer structure of extreme point sets, and instead to study how closeness to a round disk with respect to the area of the set may depend, in a family of non-polygonal examples, on a parameter. As a family of choice we take a variant of Kwapisz's construction, for which the irrational number $\rho$ of the Denjoy dynamics on the separate circles serves as the parameter. Let us recall Kwapisz's result.
\begin{theorem}[Kwapisz, 1994 \cite{Kwapisz2}]
Let $\rho\in(0, 1)$ be irrational. There is a $C^1$-diffeomorphism $G:\mathbb{T}^2\rightarrow\mathbb{T}^2$ homotopic to identity whose rotation set $\rho(G)$ equals 

$$\Lambda_\rho:=\text{Conv}\left(\left\{(0, \rho), (\rho, 0)\right\}\cup\left\{\left(\tfrac{\lceil m\rho\rceil}{m+n+1}, \tfrac{\lceil n\rho\rceil}{m+n+1}\right)\,|\,m, n\in\mathbb{N}, \alpha_m^\rho<\rho, \alpha_n^\rho<\rho\right\}\right)$$ 
where $\alpha^\rho_m:=\lceil m\rho\rceil-m\rho$ for any $m\in\mathbb{N}$. $\Lambda_\rho$ has infinitely many extreme points and exactly two accumulation points of extreme points, $(0, \rho)$ and $(\rho, 0)$.
\end{theorem}
We study the roundness of rotation sets of a closely related parametric family of torus diffeomorphisms $F_\rho$, where the parameter $\rho$ ranges over irrational numbers in $(0,1)$. Each $F_\rho$ is a Kwapisz-like diffeomorphism with a 2-dimensional non-polygonal rotation set,  whose extreme point set contains exactly four (two-sided) accumulation points. These diffeomorphisms are obtained by a small modification to Kwapisz's original construction (described in the Appendix of the present article), whose aim is to place the rotation sets in all four quadrants, by adding symmetric images of the Kwapisz's original rotation sets. This results in rotation sets that can better fill a round disk of radius $\rho$. 
  \begin{theorem}\label{four}
For each $\rho\in(0,1)\setminus \mathbb{Q}$ there exists a $C^1$-diffeomorphism $F_\rho:\mathbb{T}^2\to\mathbb{T}^2$ homotopic to the identity, whose rotation set is non-polygonal and given by
$$\Lambda'_\rho = \operatorname{conv}\left\{(\pm\frac{\lceil m\rho \rceil}{m+n+1}, \pm\frac{\lceil n\rho \rceil}{m+n+1})\,|\, m, n \in \mathbb{N} _0, \lceil m\rho\rceil - m\rho<\rho,\lceil n\rho\rceil - n\rho<\rho\right\}$$
\end{theorem}
  We define the roundness of $\Lambda'_\rho$ as the ratio $R_\rho=\frac{\operatorname{Area}(\Lambda'_\rho)}{\pi\rho^2}$, and give its upper and lower bounds in terms of $\rho$. 
  \begin{figure}[h]
\begin{center}
\includegraphics[width=0.48\textwidth]{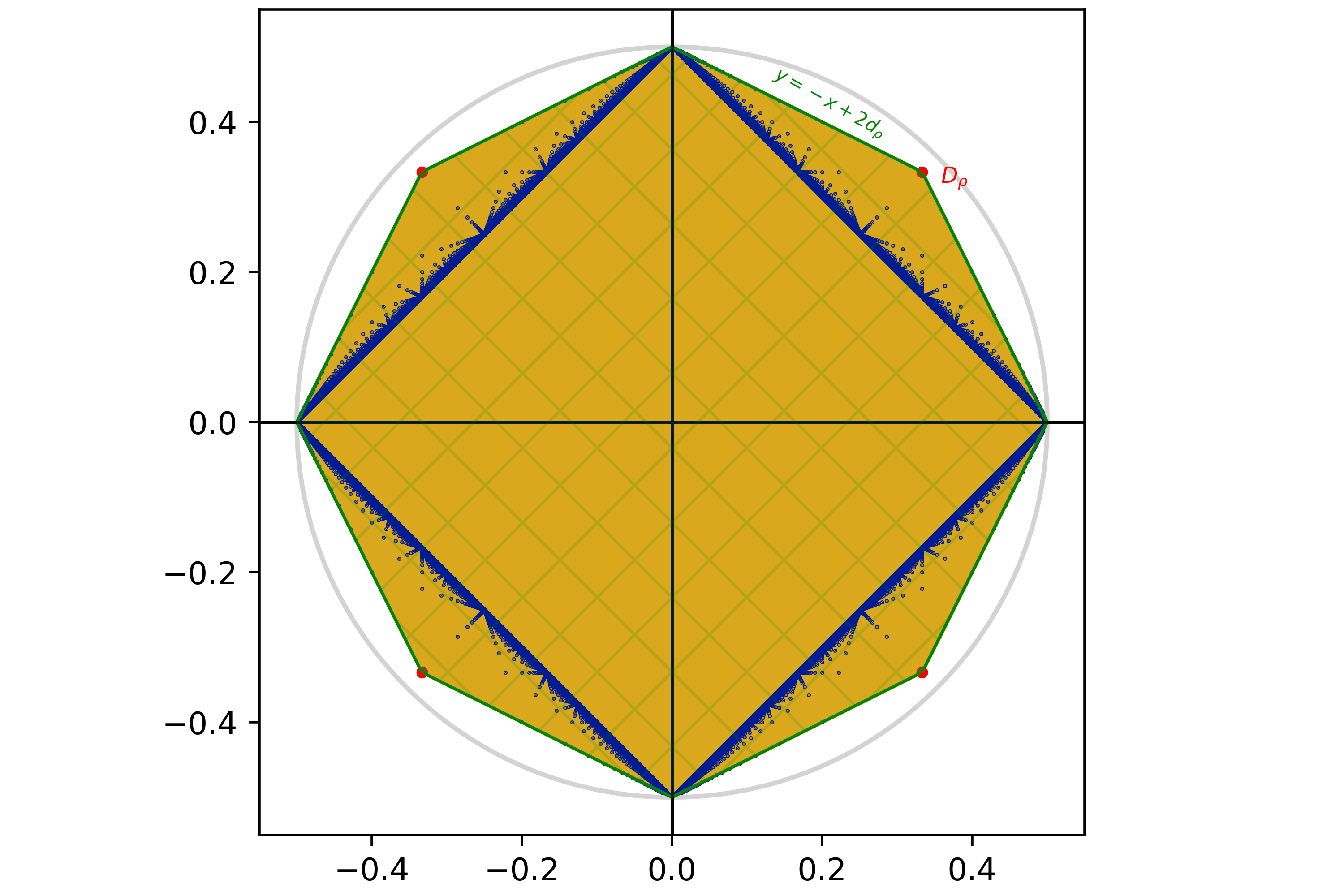}
\includegraphics[width=0.48\textwidth]{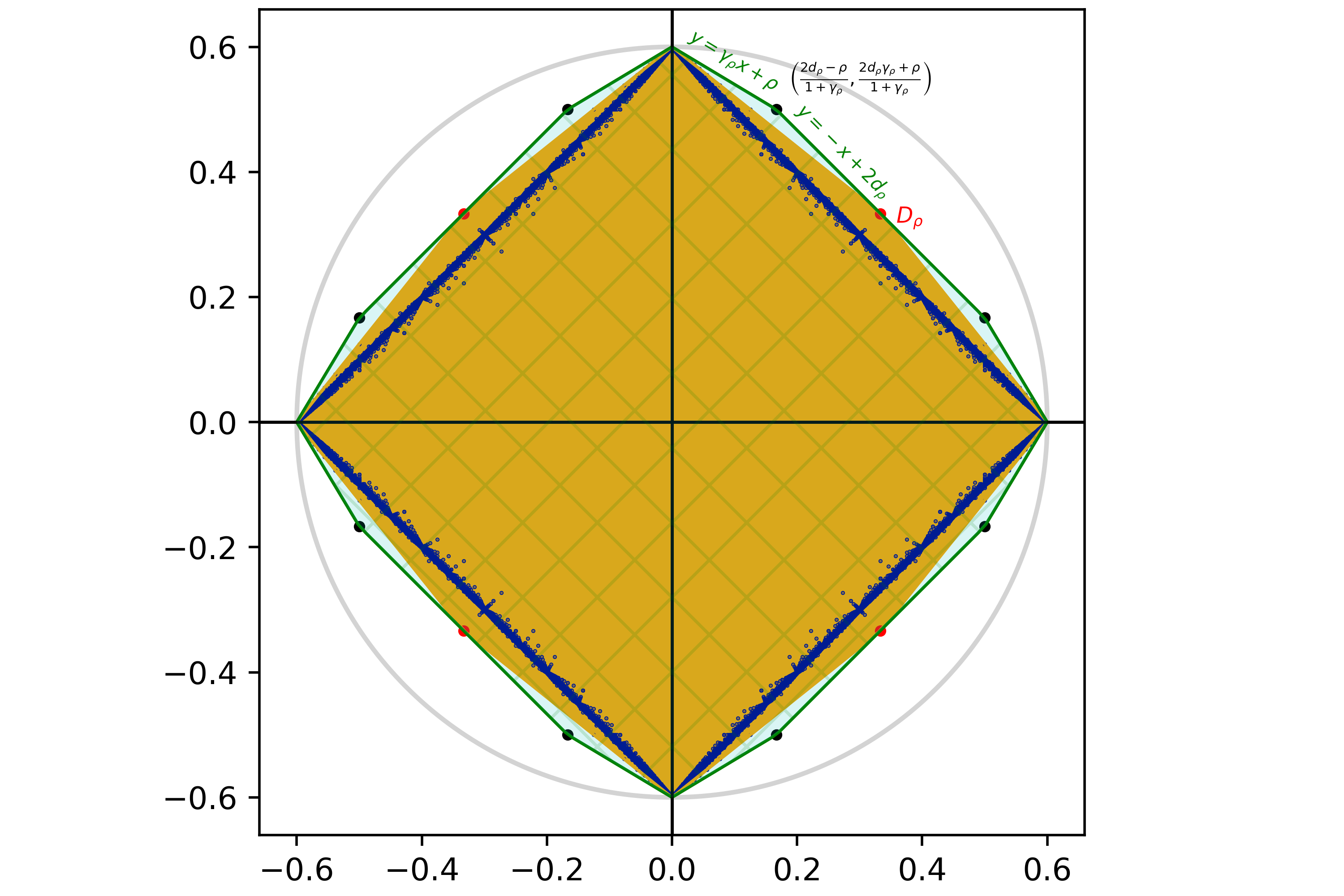}
\begin{minipage}[c]{0.8\textwidth}
\begin{center}
\caption{An approximation of the non-polygonal rotation set $\Lambda'_\rho$ for $\rho=0.5^+$ (left) and $\rho=0.6^-$ (right) }\label{pic:f}
\end{center}
\end{minipage}
\end{center}
\end{figure}
  \begin{theorem}[Bounds on roundness of $\Lambda'_\rho$]\label{thm:bounds on roundness}
Let $d_\rho=\left(\left\lfloor\tfrac{1}{1-\rho}\right\rfloor-1\right)\left(2\left\lfloor\tfrac{1}{1-\rho}\right\rfloor-1\right)^{-1}$. Then
$$\tfrac{4}{\pi}\tfrac{d_\rho}{\rho}\leqslant R_\rho\leqslant \tfrac{4}{\pi}\tfrac{1}{1-\rho\left\lfloor\tfrac{1}{\rho}\right\rfloor}\left(-1+4\tfrac{d_\rho}{\rho}-2\left(1+\rho\left\lfloor\tfrac{1}{\rho}\right\rfloor\right)\tfrac{d_\rho^2}{\rho^2}\right)$$
\end{theorem}
Note that the two bounds give $\lim_{\rho\to \frac{k-1}{k}^+} R_\rho=\frac{4k}{\pi(2k-1)}$ for any $k\in\mathbb{N}\setminus\{0\}$, and $$\sup_{\rho\in(\frac{1}{2},1)} R_\rho=\lim_{\rho\to\frac{1}{2}^+} R_\rho=\frac{8}{3\pi}, \inf_{\rho\in(\frac{1}{2},1)} R_\rho=\lim_{\rho\to\frac{2}{3}^-} R_\rho=\lim_{\rho\to 1^-} R_\rho=\frac{2}{\pi}.$$ $R_\rho$ is neither monotone nor continuous. 

In comparison, roundness in the sense of the International Organization for Standardization of the above sets $R^{ISO}_\rho$, given as the ratio of the radius of the inscribed circle to the radius of the circumscribed circle, is given by
$$R^{ISO}_\rho=\frac{d_\rho\sqrt{2}}{\rho}.$$
It is proportional to the lower bound of $R_\rho$ and piecewise decreasing, with monotonicity intervals $\{(\frac{k+1}{k+2},\frac{k+2}{k+3}):k\in\mathbb{N}\}.$

\subsection{Summary}
Section \ref{sec:prelim} contains preliminaries. In Section \ref{sec:1}, we study and determine the set $I_\rho$ of indeces $n\in\mathbb{N}\setminus\left\{0\right\}$ such that $\alpha_n^\rho<\rho$ to eliminate the condition "$\alpha_n^\rho<\rho$", which makes the structure of the set $\Lambda_\rho$ complicated to understand. Then, in Section \ref{sec:2}, we study diagonal points $A^\rho_{m, m}$ that are part of $\Lambda_\rho$ and determine the maximal value of both coordinates for such a point. The point for which this maximum is attained will be called \textit{ the best diagonal point}. In Section \ref{sec:3}, we determine whether the \text{best diagonal point} is an extreme point of the set $\Lambda_\rho$ or not. Finally, in Section \ref{sec:4}, we introduce the notion of \textit{roundness} and obtain lower and upper bounds for roundness of the set $\Lambda_\rho$. The Appendix contains the description of the modification in the original Kwapisz's construction that is needed in order to make the rotation set symmetric about the coordinate axes and the origin, which also results in the first examples of rotation sets with more than two accumulation extreme points. 

\section{Preliminaries}\label{sec:prelim}
\subsection{Rotation theory}
Rotation theory studies the asymptotic behavior of orbits in terms of their rotation vectors, which indicate the average speed and direction of these orbits as they move around their phase space. Its roots go back to celestial mechanics and observation of motions of planets. In 1885 Poincar\'e defined the rotation number of a circle automorphism, in reference to the change of rotational parameters of a planetary orbit, as the asymptotic average displacement of an arbitrary point. His work, complemented by the results of Denjoy, provided a complete classification of dynamics of orientation-preserving circle diffeomorphisms and s characterization of the asymptotic behavior of their orbits in terms of the rotation number. This provided one of the first classification results in dynamical systems.

The transition from the one-dimensional circle to higher dimensions did not occur until the late 1980's, when Misiurewicz and Ziemian \cite{MZ} introduced a systematic approach to the concept of the rotation set of a torus homeomorphism of arbitrary dimension homotopic to the identity (see also \cite{KMG} and \cite{LM}). As one would expect, beyond dimension one there is no longer a single number that, for a given torus homeomorphism, can capture the asymptotic behavior of all orbits of points traveling around the torus, but one has to consider an entire set of vectors that represent the speed and direction of orbits. This led Misiurewicz and Ziemian to define the rotation set as the set of accumulation points of average displacements of (possibly distinct) points on the torus over arbitrary finite parts of their orbits. For the 2-dimensional torus they proved that the rotation set is always compact and convex, which suggests that one could find again a classification of torus homeomorphisms dependent on the rotation set. This, however, turned out to be quite challenging. For the case when the rotation set is 1-dimensional, Franks and Misiurewicz formulated a conjecture about the possible properties of such a set, according to which it must either be a line segment with rational slope that contains infinitely many rational points, or line segment with irrational slope that has one rational endpoint. The conjecture was disproved around 2014 by A. Avila, who constructed an example with irrational slope and no rational endpoint. A lot of progress has been made for the case of rational slope, where the conjecture was confirmed for minimal homeomorphisms \cite{Kocsard},\cite{KPS}. If the rotation set is 2-dimensional even less is known about its geometric features. In the early 1990's Kwapisz showed that any polygon with rational vertices can appear as the rotation set \cite{Kwapisz1}. He also constructed an example of a rotation set which is not a polygon. This set has two extreme points which are partially irrational, and are the limits of rational extreme points \cite{Kwapisz2}. A related example was given by Beguin, Crovisier and Le Roux \cite[Proposition 2.13]{BCL}. In general, the complete picture is still very unclear and many fundamental questions remain open (see e.g. \cite{ABP},\cite{BCL},\cite{BCJL},\cite{JKT},\cite{KT1},\cite{KT2},\cite{LT1},\cite{LT2}, and \cite{L}). Applications of rotation theory are found in the theory of billiards, KAM theory, combinatorial dynamics, complex dynamics, asymptotic homology and chaos and bifurcation theory. In 2016 Boyland, de Carvalho and Hall demonstrated a new approach to study the geometry of rotation sets, through inverse limits of parametric families of 1-dimensional maps \cite{BCH1}. In \cite{BCH2}, aided by methods from symbolic dynamics \cite{BCH3}, they provided the first instance of a systematically studied parametric family of rotation sets of torus homeomorphisms. These sets included non-polygonal rotation sets with an extreme point which is totally irrational and accumulated on both sides by other extreme points, improving on the result of Kwapisz. The paper also contained the first example of a parametric family of torus homeomorphisms with  the Tal-Zanata property, i.e., discontinuity in the sets of extreme points.

\subsection{Case of the 2-dimensional torus $\mathbb{T}^2$}
The notion of \textit{rotation set} in dimension 2, which is pertinent here, was introduced in \cite{MZ}. The rotation set of a lift may be defined in several different ways.

\begin{definition}
Let $f:\mathbb{T}^2\rightarrow\mathbb{T}^2$ be a homeomorphism homotopic to the identity, and $F:\mathbb{R}^2\rightarrow\mathbb{R}^2$ a lift of $f$. For any $x\in\mathbb{T}^2$, let $\tilde{x}$ denote a lift of $x$. Define $\varphi:\mathbb{T}^2\rightarrow\mathbb{R}^2$ by $\varphi(x) = F\left(\tilde{x}\right)-\tilde{x}$ for any $x\in\mathbb{T}^2$. Then let

$$\begin{aligned}\rho_{\text{MZ}}\left(F\right) &= \bigcap_{m>0}\text{cl}\left(\bigcup_{n>m}\left\{\tfrac{1}{n}\left(F^n\left(\tilde{x}\right)-\tilde{x}\right):\tilde{x}\in\mathbb{R}^2\right\}\right) \\
%&\text{where notation "cl" means "closure of"} \\
\rho_{\text{p}}\left(F\right) &= \left\{\underset{n\to\infty}\lim\tfrac{1}{n}\left(F^n\left(\tilde{x}\right)-\tilde{x}\right):x\in\mathbb{T}^2, \text{ the limit exists}\right\} \\
\rho_{\text{m}}\left(F\right) &= \left\{\int_{\mathbb{T}^2} \varphi\,\text{d}\mu:\mu \,f\text{-invariant probability measure on }\mathbb{T}^2\right\}
\end{aligned}$$

\noindent denote respectively the \textit{Misiurewicz-Ziemian rotation set of $F$}, the \textit{pointwise rotation set of $F$} and the \textit{measure rotation set of $F$}. 
\end{definition}

Misiurewicz and Ziemian \cite{MZ} showed that these sets have the same convex hulls.

\begin{theorem}[Misiurewicz-Ziemian, 1989]
Let $f:\mathbb{T}^2\rightarrow\mathbb{T}^2$ be a homeomorphism homotopic to the identity, and $F:\mathbb{R}^2\rightarrow\mathbb{R}^2$ a lift of $f$. Then 

$$\rho_{\text{MZ}}\left(F\right) = \text{Conv}\,\rho_{\text{p}}\left(F\right) = \text{Conv}\rho_{\text{m}}\left(F\right)$$
\end{theorem}

Therefore, it is convenient to define the rotation set of a map of the torus in the following way. 

\begin{definition}
Let $f:\mathbb{T}^2\rightarrow\mathbb{T}^2$ be a be a homeomorphism homotopic to the identity, and $F:\mathbb{R}^2\rightarrow\mathbb{R}^2$ a lift of $f$. Define the \textit{rotation set of map $f$} as 

$$\rho(f)=\rho_{\text{MZ}}\left(F\right)+\mathbb{Z}^2$$
\end{definition}

A natural question is then which compact convex subsets of the plane can be realized as a rotation set of a homeomorphism of the torus. The aforementioned question on the round disk is a particularly interesting special case.

\subsection{Notation}
Let us introduce our notation. Let $\rho\in(0, 1)\setminus\mathbb{Q}$.

\begin{definition}\label{def:alpha_n}
For any $m\in\mathbb{N}$, let 

$$\begin{aligned}
\alpha^\rho_m &:= \lceil m\rho\rceil - m\rho \\
&\phantom{:}= 1-D(m\rho)
\end{aligned}$$

where $D$ is the decimal part.

\end{definition}

Note that the sequence $\left(\alpha_m^\rho\right)_{m\in\mathbb{N}}$ is dense in $[0, 1]$ and all its terms are distinct.

\begin{definition}\label{def:Lambda_rho}
For any $m, n\in\mathbb{N}$, denote

$$A^\rho_{m, n} = (x_{m, n}^\rho, y_{m, n}^\rho) := \left(\tfrac{\lceil m\rho\rceil}{m+n+1}, \tfrac{\lceil n\rho\rceil}{m+n+1}\right)$$

Then denote by $\Lambda_\rho$ the convex hull \footnote{meaning the intersection of all convex sets containing \\ $\left\{A_{m, n}^\rho\,|\,m, n\in\mathbb{N}, \alpha_m<\rho, \alpha_n<\rho\right\}\cup\left\{(0, \rho), (\rho, 0)\right\}$} of

$$\left\{A_{m, n}^\rho\,|\,m, n\in\mathbb{N}, \alpha_m^\rho<\rho, \alpha_n^\rho<\rho\right\}\cup\left\{(0, \rho), (\rho, 0)\right\},$$

which is a compact and convex subset of $\mathbb{R}^2$.  
\end{definition}

Note that since $(0, 0), (0, \rho), (\rho, 0)\in\Lambda_\rho$,  the set $\Lambda_\rho$ can also be seen as the convex hull of 

$$L_\rho:=\left\{A_{m, n}^\rho\,|\,m, n\in\mathbb{N}\setminus\left\{0\right\}, \alpha_m^\rho<\rho, \alpha_n^\rho<\rho\right\}\cup\left\{(0, 0), (0, \rho), (\rho, 0)\right\}$$

\noindent because points $A_{m, n}^\rho$ (satisfying $\alpha_m^\rho<\rho, \alpha_n^\rho<\rho$) with $m=0$ or $n=0$ are (non-trivial) convex combinations of $(0, 0), (0, \rho), (\rho, 0)$. Furthermore, note that the extreme points of $\Lambda_\rho$ are elements of $L_\rho$. Finally, for any $m, n\in\mathbb{N}$, we have 

$$A_{m, n}^\rho\in \Lambda_\rho \iff A_{n, m}^\rho\in \Lambda_\rho$$

\noindent and hence, the set $\Lambda_\rho$ is symmetric with respect to the diagonal.

\begin{definition}\label{def:I_rho}
We introduce 

$$I_\rho:=\left\{m\in\mathbb{N}\setminus\left\{0\right\}\,|\,\alpha_m^\rho<\rho\right\}$$

so that 

$$L_\rho = \left\{A_{m, n}^\rho\,|\,m, n\in I_\rho\right\}\cup\left\{(0, 0), (0, \rho), (\rho, 0)\right\}$$
\end{definition}

Note the following symmetry property 

\begin{proposition}\label{prop:symmetry} 
For all $m, n\in\mathbb{N}\setminus\left\{0\right\}$, one has 

$$\begin{aligned}
\alpha_{m}^{1-\rho} &= 1-\alpha_{m}^\rho \\
A^{1-\rho}_{m, n} &= -A^\rho_{m, n}+\left(\tfrac{m+1}{m+n+1}, \tfrac{n+1}{m+n+1}\right) \\ 
\end{aligned}$$
\end{proposition}

\begin{proof}
Recall that 

$$\forall x\in\mathbb{R}, \forall k\in\mathbb{Z}, \left\{\begin{aligned}\lceil x+k\rceil &= \lceil x\rceil+k \\ \lceil-x\rceil &= -\lceil x\rceil+1 \text{ if }x\notin\mathbb{Z}\end{aligned}\right.$$

Using these two properties and the fact that for any $m\in\mathbb{N}\setminus\left\{0\right\}, m\rho\notin\mathbb{Z}$, a simple computation leads to the result.
\end{proof}

\section{Study of $I_\rho$}\label{sec:1}

Since extreme points of the set $\Lambda_\rho$ are part of $L_\rho$, it is relevant to study $L_\rho$. In this section, we determine the set $I_\rho$ in order to eliminate the condition $\alpha_n^\rho<\rho$ in the definition of set $L_\rho$. In this way, we obtain an explicit description of $L_\rho$.

\subsection{Case $\rho>\frac{1}{2}$}

In this subsection, we assume that $\rho>\tfrac{1}{2}$.

\begin{definition}\label{def:partitions of I_rho and N*, >0.5}
Let $k\in\mathbb{N}$, and denote 

$$I_{\rho, k}=\left\{n\in I_\rho\,|\,\lceil n\rho\rceil = n-k\right\} \text{ and }J_{\rho, k}=\left\{n\in \mathbb{N}\setminus\left\{0\right\}\,|\,\lceil n\rho\rceil = n-k\right\}$$
\end{definition}

Note that by definition, the sets $(I_{\rho, k})_{k\in\mathbb{N}}$ establish a partition of $I_\rho$, and the sets $(J_{\rho, k})_{k\in\mathbb{N}}$ establish a partition of $\mathbb{N}\setminus\left\{0\right\}$. Also, for any $k\in\mathbb{N}$, $I_{\rho, k}\subset J_{\rho, k}$. 

\begin{definition}\label{def:maxima partition, >0.5}
For any $k\in\mathbb{N}\cup\left\{-1\right\}$, we denote 

$$N_{\rho, k}:=\left\lfloor\tfrac{k+1}{1-\rho}\right\rfloor-1$$
\end{definition}

\begin{proposition}\label{prop:partition, >0.5}
The following defines a partition of $I_\rho$

$$\begin{aligned}I_\rho &= \bigcup_{k=0}^\infty I_{\rho, k} \text{ with } I_{\rho, k}=\left\llbracket N_{\rho, k-1}+2, N_{\rho, k}\right\rrbracket \\
&= \mathbb{N}\setminus\left\{\left\lfloor\tfrac{k}{1-\rho}\right\rfloor\,|\,k\in\mathbb{N}\right\} \\
\end{aligned}$$

such that for any $k, l\in\mathbb{N}$ and $m\in I_{\rho, k}, l\in I_{\rho, l}$, we have

$$\alpha_m^\rho = m-k-m\rho \text{ and }\alpha_n^\rho = n-l-n\rho$$

and therefore

$$x_{m, n}^\rho = \tfrac{m-k}{m+n+1} \text{ and }y_{m, n}^\rho = \tfrac{n-l}{m+n+1}$$

The sequence $\left(\alpha_n^\rho\right)_{n\in\mathbb{N}\setminus\left\{0\right\}}$ is strictly increasing on each interval $I_{\rho, k}$.

\end{proposition}

Note that since $\rho>\tfrac{1}{2}$, for any $k\in\mathbb{N}$, we have $\text{Card}\left(J_{\rho, k}\right)\geqslant2$. 

\begin{figure}[!h]
\begin{center}
\includegraphics[width=11cm]{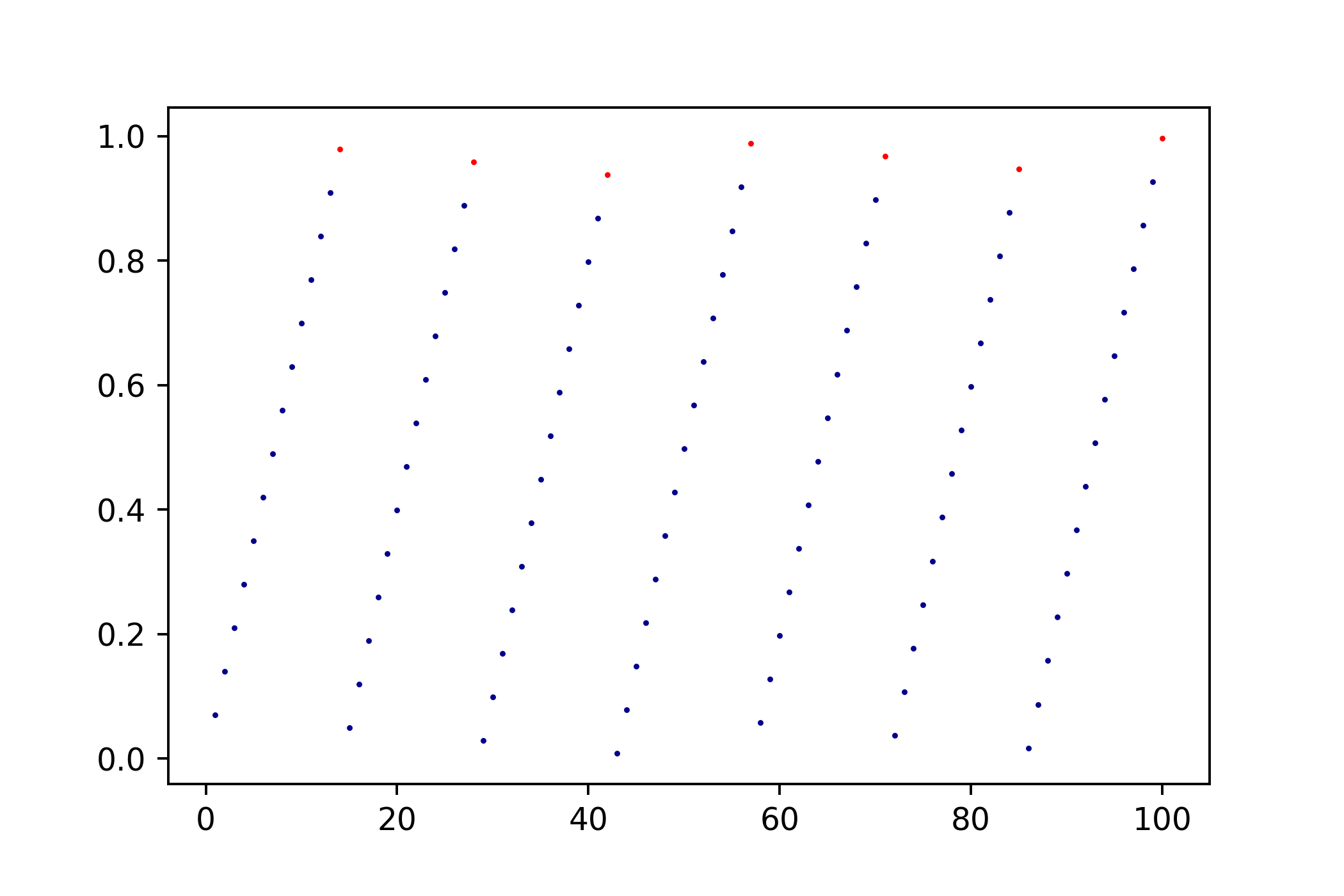}
\end{center}
    \vskip-1cm\caption{First 100 terms of sequence $(\alpha_{n}^\rho)_{n \in \mathbb{N}\setminus\left\{0\right\}}$ for $\rho=0.93+\pi.10^{-5}$. Red points are such that $n\notin I_\rho$ and blue ones are such that $n\in I_\rho$.}
    %\label{fig:Fig. 7}
\end{figure}

\begin{proof}
Let $k\in\mathbb{N}$.

\begin{itemize}

\item First, note that since $\rho>\tfrac{1}{2}$, we have $N_{\rho, k}=\left\lfloor\tfrac{k+1}{1-\rho}\right\rfloor-1\in\mathbb{N}\setminus\left\{0\right\}$. 

\item Second, for any $n\in\mathbb{N}\setminus\left\{0\right\}$,

$$\begin{aligned}
&n\in J_{\rho, k} \\
\iff&\lceil n\rho\rceil = n-k \\
\iff&n\rho\in\left(n-k-1, n-k\right) \text{ since }n\rho\notin\mathbb{Z} \\
\iff&n\in\left(\tfrac{k}{1-\rho}, \tfrac{k+1}{1-\rho}\right) \\
\iff&n\in\left(\left\lfloor\tfrac{k}{1-\rho}\right\rfloor+1, \left\lfloor\tfrac{k+1}{1-\rho}\right\rfloor-1\right) \text{ since }\left\lceil\tfrac{k}{1-\rho}\right\rceil=\left\lfloor\tfrac{k}{1-\rho}\right\rfloor+1
\end{aligned}$$

\item Let $n\in\mathbb{N}\setminus\left\{0, 1\right\}$. Suppose that $\alpha_{n-1}^\rho<\rho$ and $\alpha_{n+1}^\rho<\rho$. Let's show that $\alpha_n^\rho<\rho$. Indeed, by assumption, we have 

$$\begin{aligned}
\alpha_{n-1}^\rho &= n-1-k-(n-1)\rho<\rho \\
\alpha_{n+1}^\rho &= n+1-k-(n+1)\rho<\rho
\end{aligned}$$

Thus, by taking the mean of those two lines, one obtains 

$$\alpha_n^\rho=n-k-n\rho<\rho$$

This shows that $I_{\rho, k}$ is an interval, since $J_{\rho, k}$ is an interval (according to the previous point). Therefore, we just need to prove

$$\begin{aligned}
\max I_{\rho, k}&= \left\lfloor\tfrac{k+1}{1-\rho}\right\rfloor-1 \\
\min I_{\rho, k}&= \left\lfloor\tfrac{k}{1-\rho}\right\rfloor+1
\end{aligned}$$ 

\item We have

$$\begin{aligned}
\alpha^\rho_{\left\lfloor\tfrac{k+1}{1-\rho}\right\rfloor-1}&=\left\lfloor\tfrac{k+1}{1-\rho}\right\rfloor-1-k-\left(\left\lfloor\tfrac{k+1}{1-\rho}\right\rfloor-1\right)\rho \text{ since }\left\lfloor\tfrac{k+1}{1-\rho}\right\rfloor-1\in J_{\rho, k} \\
&=\left(\left\lfloor\tfrac{k+1}{1-\rho}\right\rfloor-1\right)(1-\rho)-k \\
&<\left(\tfrac{k+1}{1-\rho}-1\right)(1-\rho)-k \text{ since }\tfrac{k+1}{1-\rho}\notin\mathbb{Z} \\
&=\rho \\
\end{aligned}$$

$$\begin{aligned}
\alpha^\rho_{\left\lfloor\tfrac{k+1}{1-\rho}\right\rfloor}&=\left\lfloor\tfrac{k+1}{1-\rho}\right\rfloor-k-\left\lfloor\tfrac{k+1}{1-\rho}\right\rfloor\rho \text{ since }\left\lfloor\tfrac{k+1}{1-\rho}\right\rfloor\in J_{\rho, k}\phantom{....................} \\
&=\left\lfloor\tfrac{k+1}{1-\rho}\right\rfloor(1-\rho)-k \\
&>\left(\tfrac{k+1}{1-\rho}-1\right)(1-\rho)-k \\
&=\rho \\
\end{aligned}$$

$$\begin{aligned}
\alpha_{\left\lfloor\tfrac{k}{1-\rho}\right\rfloor+1}^\rho &= \left\lfloor\tfrac{k}{1-\rho}\right\rfloor+1-k-\left(\left\lfloor\tfrac{k}{1-\rho}\right\rfloor+1\right)\rho \text{ since }\left\lfloor\tfrac{k}{1-\rho}\right\rfloor+1\in J_{\rho, k} \\
&= \left(\left\lfloor\tfrac{k}{1-\rho}\right\rfloor+1\right)(1-\rho)-k \\
&< \left(\tfrac{k}{1-\rho}+1\right)(1-\rho)-k \\
&=1-\rho \\
&<\rho \text{ since }\rho>\tfrac{1}{2}
\end{aligned}$$

\item In other words, we proved that 

$$\left\lfloor\tfrac{k+1}{1-\rho}\right\rfloor-1, \left\lfloor\tfrac{k}{1-\rho}\right\rfloor+1\in I_{\rho, k} \text{ and }\left\lfloor\tfrac{k+1}{1-\rho}\right\rfloor\notin I_{\rho, k}$$

Consequently,  

$$\begin{aligned}
\max I_{\rho, k}&=\max J_{\rho, k}-1 = \left\lfloor\tfrac{k+1}{1-\rho}\right\rfloor-1 \\
\min I_{\rho, k}&=\min J_{\rho, k} = \left\lfloor\tfrac{k}{1-\rho}\right\rfloor+1 \text{ since }I_{\rho, k}\subset J_{\rho, k}
\end{aligned}$$ 

\item The sequence $(\alpha_{n}^\rho)_{n\in \mathbb{N}\setminus\left\{0\right\}}$ is strictly increasing on the interval $J_{\rho, k}$. Indeed, for any $n\in J_{\rho, k}$ such that $n+1\in J_{\rho, k}$, one has 

$$\begin{aligned}
\alpha_{n+1}^\rho-\alpha_n^\rho &= \lceil(n+1)\rho\rceil-(n+1)\rho-\lceil n\rho\rceil+n\rho \\
&=n+1-k-(n+1)\rho-(n-k)+n\rho \\
&=1-\rho \\
&>0
\end{aligned}$$
\end{itemize}
\end{proof}

\subsection{Case $\rho<\frac{1}{2}$}

In this subsection, we assume that $\rho<\tfrac{1}{2}$. The partition $(I_{\rho, k})_{k\in\mathbb{N}}$ of $I_\rho$ introduced in Definition \ref{def:partitions of I_rho and N*, >0.5} is not appropriate in this case, because some intervals $I_{\rho, k}$ are empty. Therefore, we introduce another, more relevant partition.  

\begin{definition}\label{def:partitions of I_rho and N*, <0.5}
Let $k\in\mathbb{N}\setminus\left\{0\right\}$, and denote 

$$I_{\rho, k}'=\left\{n\in I_\rho\,|\,\lceil n\rho\rceil = k\right\} \text{ and }J_{\rho, k}'=\left\{n\in \mathbb{N}\setminus\left\{0\right\}\,|\,\lceil n\rho\rceil = k\right\}$$
\end{definition}

Note that by definition, the sets $\left(I_{\rho, k}'\right)_{k\in\mathbb{N}\setminus\left\{0\right\}}$ establish a partition of $I_\rho$ and $\left(J_{\rho, k}'\right)_{k\in\mathbb{N}\setminus\left\{0\right\}}$a partition of $\mathbb{N}\setminus\left\{0\right\}$. Also, for any $k\in\mathbb{N}\setminus\left\{0\right\}$, $I_{\rho, k}'\subset J_{\rho, k}'$.

\begin{lemma}\label{lem:explicit expression for partitions, <0.5}
For any $k\in\mathbb{N}\setminus\left\{0\right\}$, one has 

$$\begin{aligned}
I_{\rho, k}'&=\left\{\left\lfloor\tfrac{k}{\rho}\right\rfloor\right\} \\
J_{\rho, k}'&=\left\llbracket\left\lfloor\tfrac{k-1}{\rho}\right\rfloor+1 , \left\lfloor\tfrac{k}{\rho}\right\rfloor\right\rrbracket
\end{aligned}$$
\end{lemma}

\begin{figure}[!h]
\begin{center}
\includegraphics[width=11cm]{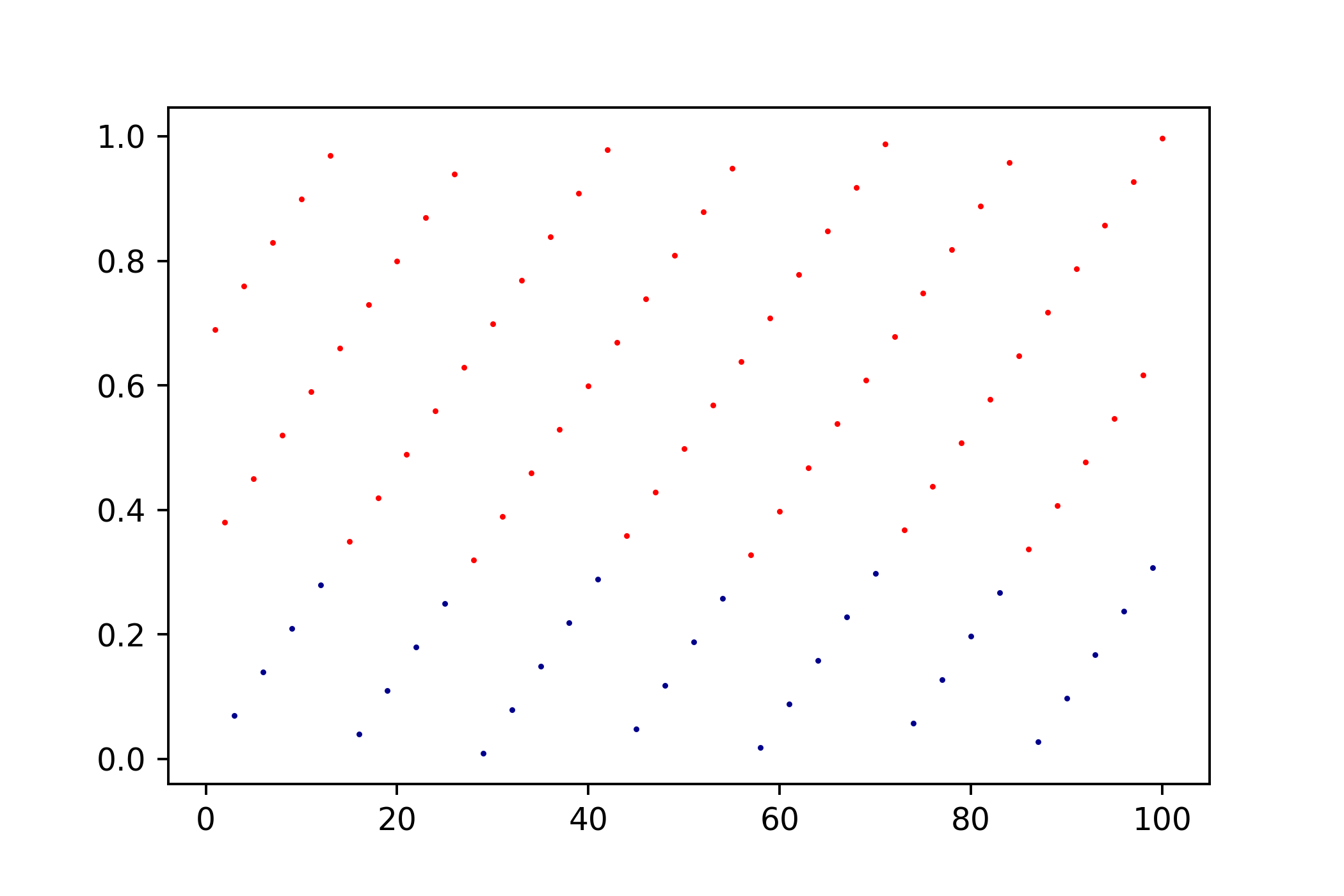}
\end{center}
    \vskip-1cm\caption{First 100 terms of sequence $(\alpha_{n}^\rho)_{n \in \mathbb{N}\setminus\left\{0\right\}}$ for $\rho=0.31+\pi.10^{-5}$. Red points are such that $n\notin I_\rho$ and blue ones are such that $n\in I_\rho$.}
    %\label{fig:Fig. 7}
\end{figure}

\begin{proof}
Let $m, k\in\mathbb{N}\setminus\left\{0\right\}$.
\begin{itemize}
\item According to Proposition \ref{prop:symmetry}, we have 

$$\begin{aligned}
&m\in I_\rho \\
\iff&\alpha_m^{\rho}<\rho \\
\iff&\alpha_m^{1-\rho}>1-\rho \\
\iff&m\notin I_{1-\rho}
\end{aligned}$$

\item Additionally 

$$\lceil m(1-\rho)\rceil = m+1-\lceil m\rho\rceil$$

hence 

$$\lceil m\rho\rceil = k \iff \lceil m(1-\rho)\rceil = m+1-k$$

\item Therefore 

$$\begin{aligned}
m\in J_{\rho, k}' &\iff m\in J_{1-\rho, k-1} \\
&\iff m\in\left\llbracket\left\lfloor\tfrac{k-1}{\rho}\right\rfloor+1 , \left\lfloor\tfrac{k}{\rho}\right\rfloor\right\rrbracket
\end{aligned}$$

and

$$\begin{aligned}
m\in I_{\rho, k}' &\iff m\in J_{1-\rho, k-1}\setminus I_{1-\rho} \\
&\iff m=\left\lfloor\tfrac{k}{\rho}\right\rfloor
\end{aligned}$$

\end{itemize}
\end{proof}

Note the "symmetry" property which appeared in the first point of this proof.

$$I_{\rho}=\left(\mathbb{N}\setminus\left\{0\right\}\right)\setminus I_{1-\rho}$$

Since $\rho<\tfrac{1}{2}$, for any $k\in\mathbb{N}$, we have $\text{Card}\left(J_{\rho, k}'\right)\geqslant2$. In order to have a more precise expression for points $A_{m, n}^\rho$ with indices $m, n\in I_\rho$, we regroup points of $I_\rho$ thanks to the following new partition.

\begin{definition}\label{def:new partition, <0.5}
In the following, for any $k\in\mathbb{N}\setminus\left\{0\right\}$, we denote 

$$N_{\rho, k}'=\left\lfloor\tfrac{k}{\rho}\right\rfloor$$

and for any $j\in\mathbb{N}$, denote
$$\begin{aligned}
M_{\rho, j}&=\left\lfloor\tfrac{j\rho}{1-\left\lfloor\tfrac{1}{\rho}\right\rfloor\rho}\right\rfloor \\
Y_{\rho, j}&=\left\{l\in\mathbb{N}\setminus\left\{0\right\}\,|\,(j-1)\rho\leqslant ls_\rho<j\rho\right\} \text{ for }j\geqslant1
\end{aligned}$$

where

$$s_\rho:=\alpha_{N_{\rho, 1}'}^\rho = 1-\left\lfloor\tfrac{1}{\rho}\right\rfloor\rho$$

\end{definition}

One can see $s_\rho$ as a "step" for the sequence $\left(\alpha^\rho_{N_{\rho, k}'}\right)_{k\in\mathbb{N}\setminus\left\{0\right\}}$, and we will show that this sequence is strictly increasing on each set $Y_{\rho, j}$. Note that $\left(Y_{\rho, j}\right)_{j\in\mathbb{N}\setminus\left\{0\right\}}$ defines a partition of $\mathbb{N}\setminus\left\{0\right\}$.

\begin{lemma}\label{lem:explicit expression new partition, <0.5}
For any $j\in\mathbb{N}\setminus\left\{0\right\}$, one has

$$Y_{\rho, j}=\left\llbracket M_{\rho, j-1}+1, M_{\rho, j}\right\rrbracket$$
\end{lemma}

\begin{proof}
Let $j, l\in\mathbb{N}\setminus\left\{0\right\}$. One has 

$$\begin{aligned}
&l\in Y_{\rho, j} \\
\iff& (j-1)\rho\leqslant ls_\rho<j\rho \\
\iff& \tfrac{(j-1)\rho}{1-\left\lfloor\tfrac{1}{\rho}\right\rfloor\rho}\leqslant l<\tfrac{j\rho}{1-\left\lfloor\tfrac{1}{\rho}\right\rfloor\rho} \\
\iff& \left\lceil\tfrac{(j-1)\rho}{1-\left\lfloor\tfrac{1}{\rho}\right\rfloor\rho}\right\rceil\leqslant l \leqslant\left\lfloor\tfrac{j\rho}{1-\left\lfloor\tfrac{1}{\rho}\right\rfloor\rho}\right\rfloor \text{ since }\tfrac{j\rho}{1-\left\lfloor\tfrac{1}{\rho}\right\rfloor\rho}\notin\mathbb{Z} \\
\iff&M_{\rho, j-1}+1\leqslant l \leqslant M_{\rho, j} \text{ if }j\geqslant2
\end{aligned}$$

which still holds true for $j=1$.
\end{proof}

\begin{proposition}\label{prop:partition, <0.5}
The following defines a partition of $\mathbb{N}\setminus\left\{0\right\}$

$$\mathbb{N}\setminus\left\{0\right\} = \bigcup_{j=1}^\infty Y_{\rho, j} \text{ with } Y_{\rho, j}=\left\llbracket M_{\rho, j-1}+1, M_{\rho, j}\right\rrbracket$$

One has 

$$I_\rho = \left\{N_{\rho, k}'\,|\,k\in\mathbb{N}\setminus\left\{0\right\}\right\}$$

and for any $i, j\in\mathbb{N}\setminus\left\{0\right\}$ and $k\in Y_{\rho, i}, l\in Y_{\rho, j}$, we have

$$\alpha_{N_{\rho, k}'}^\rho=ks_\rho-(i-1)\rho\text{ and }\alpha_{N_{\rho, l}'}^\rho=ks_\rho-(j-1)\rho$$

and therefore

$$x_{N_{\rho, k}', N_{\rho, l}'}^\rho = \tfrac{k}{(k+l)\left\lfloor\tfrac{1}{\rho}\right\rfloor+i+j-1} \text{ and }y_{N_{\rho, k}', N_{\rho, l}'}^\rho = \tfrac{l}{(k+l)\left\lfloor\tfrac{1}{\rho}\right\rfloor+i+j-1}$$

The sequence $\left(\alpha_{N_{\rho, k}'}^\rho\right)_{k\in\mathbb{N}\setminus\left\{0\right\}}$ is strictly increasing on each interval $Y_{\rho, j}$.
\end{proposition}

\begin{proof}

\begin{itemize}
\item The first two claims are true according to Lemmas \ref{lem:explicit expression for partitions, <0.5} and \ref{lem:explicit expression new partition, <0.5}.
\item Let $i, j\in\mathbb{N}\setminus\left\{0\right\}$ and $k\in Y_{\rho, i}, l\in Y_{\rho, j}$. Since $N_{\rho, k}'\in I_{\rho, k}'$, one has 

$$\alpha_{N_{\rho, k}'}^\rho=k-N_{\rho, k}'\rho$$

We know that $N_{k, \rho}'=\left\lfloor\tfrac{k}{\rho}\right\rfloor$, so we need to prove that integer $k\left\lfloor\tfrac{1}{\rho}\right\rfloor+i-1$ belongs to $\left(\tfrac{k}{\rho}-1, \tfrac{k}{\rho}\right]$. One has 

$$k\left\lfloor\tfrac{1}{\rho}\right\rfloor+i-1 = \tfrac{k}{\rho}\left( 1-s_\rho\right)+i-1$$

and since $k\in Y_{\rho, j}$, we have 

$$\tfrac{i-1}{k}\rho\leqslant s_\rho<\tfrac{i}{k}\rho$$

Therefore we obtain that 

$$\tfrac{k}{\rho}-1<k\left\lfloor\tfrac{1}{\rho}\right\rfloor+i-1\leqslant\tfrac{k}{\rho}$$

Hence 

$$N_{\rho, k}'= k\left\lfloor\tfrac{1}{\rho}\right\rfloor+i-1 \text{ and }\alpha_{N_{\rho, k}'}^\rho =k-N_{\rho, k}'\rho= ks_\rho-(i-1)\rho$$

thus we obtain 

$$x_{N_{\rho, k}', N_{\rho, l}'}^\rho = \tfrac{k}{(k+l)\left\lfloor\tfrac{1}{\rho}\right\rfloor+i+j-1} \text{ and }y_{N_{\rho, k}', N_{\rho, l}'}^\rho = \tfrac{l}{(k+l)\left\lfloor\tfrac{1}{\rho}\right\rfloor+i+j-1}$$

\item Finally, according to the second point, one may notice that the sequence $\left(\alpha_{N_{\rho, k}'}^\rho\right)_{k\in\mathbb{N}\setminus\left\{0\right\}}$ is strictly increasing on each interval $Y_{\rho, j}$.
\end{itemize}
\end{proof}

\section{Best diagonal point}\label{sec:2}

Let us recall that extreme points of $\Lambda_\rho$ are elements of $L_\rho$. Here we focus on diagonal points $A_{m, m}^\rho$ belonging to $L_\rho$ :

$$\forall m\in\mathbb{N}, A_{m, m}^\rho=\left(\tfrac{\lceil m\rho\rceil}{2m+1}, \tfrac{\lceil m\rho\rceil}{2m+1}\right)$$ 

Note that 

$$x_{m, m}^\rho=y_{m, m}^\rho\underset{m\to\infty}\rightarrow\left(\tfrac{\rho}{2}, \tfrac{\rho}{2}\right)$$

We are looking for a potential diagonal extreme point of $\Lambda_\rho$ among the points of $L_\rho$ on the diagonal, that is, the one with the largest coordinates. We will call this point \textit{ the best diagonal point} of the set $\Lambda_\rho$. So we are trying to prove the existence and determine the exact form of $\underset{m\in I_\rho}\max\,x_{m, m}^\rho$.

\subsection{Case $\rho>\frac{1}{2}$}

In this subsection, we assume that $\rho>\tfrac{1}{2}$. By Proposition \ref{prop:partition, >0.5}, we have the following. 

$$\forall k\in\mathbb{N}, \forall n\in I_{\rho, k}, x_{n, n}^\rho=\tfrac{n-k}{2n+1}$$

Hence we obtain the following.

\begin{corollary}\label{coro:x_nn increasing on partition, >0.5}
Let $k\in\mathbb{N}$. The sequence $(x_{n, n}^\rho)_{n\in I_\rho}$ is strictly increasing on the interval~$I_{\rho, k}$.
\end{corollary}

\begin{lemma}\label{lem:bounds on maxima, >0.5}
For any $k\in\mathbb{N}$, one has 

$$k(N_{\rho, 0}+1)+N_{\rho, 0}\leqslant N_{\rho, k}\leqslant k(N_{\rho, 0}+2)+N_{\rho, 0}$$
\end{lemma}

\begin{proof}
Let's prove the following claim 

$$\forall k\in\mathbb{N}, N_{\rho, 0}+1\leqslant N_{\rho, k+1}-N_{\rho, k}\leqslant N_{\rho, 0}+2$$

If this claim is true, then the result holds. We will use the fact that the floor function is increasing and the following property  

$$\forall x\in\mathbb{R}, \forall l\in\mathbb{Z}, \left\lfloor x+l\right\rfloor=\left\lfloor x\right\rfloor+l$$

Let $k\in\mathbb{N}$. Then we have

$$\begin{aligned}
N_{\rho, k+1}-N_{\rho, k}&= \left\lfloor\tfrac{k+1}{1-\rho}+\tfrac{1}{1-\rho}\right\rfloor-\left\lfloor\tfrac{k+1}{1-\rho}\right\rfloor \\
&\geqslant \left\lfloor\tfrac{k+1}{1-\rho}+\left\lfloor\tfrac{1}{1-\rho}\right\rfloor\right\rfloor-\left\lfloor\tfrac{k+1}{1-\rho}\right\rfloor \\
&= \left\lfloor\tfrac{1}{1-\rho}\right\rfloor \\
&= N_{\rho, 0}+1
\end{aligned}$$

and in the same way one obtains 

$$N_{\rho, k+1}-N_{\rho, k}\leqslant N_{\rho, 0}+2$$

Thus the claim we stated at the beginning of the proof holds.
\end{proof}

We conclude with the following

\begin{theorem}[Best diagonal point, $\rho>\tfrac{1}{2}$]\label{thm:best diagonal point, >0.5}
If $\rho>\tfrac{1}{2}$, then 

$$\begin{aligned}\underset{n\in I_\rho}\max\, x_{n, n}^\rho &= x_{N_{\rho, 0}, N_{\rho, 0}}^\rho \\
&= \tfrac{N_{\rho, 0}}{2N_{\rho, 0}+1} \text{ with }N_{\rho, 0} = \left\lfloor\tfrac{1}{1-\rho}\right\rfloor-1 \\
&=\tfrac{\left\lfloor\tfrac{1}{1-\rho}\right\rfloor-1}{2\left\lfloor\tfrac{1}{1-\rho}\right\rfloor-1}
\end{aligned}$$
\end{theorem}

\begin{proof}
We start with the problem of finding $\underset{n\in I_\rho}\sup\, x_{n, n}^\rho$. Then, thanks to Corollary \ref{coro:x_nn increasing on partition, >0.5}, the problem can be reduced to finding 
$\underset{k\in\mathbb{N}}\sup\, x_{N_{\rho, k}, N_{\rho, k}}^\rho$. Finally, let's prove that this supremum is achieved and

$$\underset{k\in\mathbb{N}}\max\, x_{N_{\rho, k}, N_{\rho, k}}^\rho = x_{N_{\rho, 0}, N_{\rho, 0}}^\rho.$$

Let $k\in\mathbb{N}\setminus\left\{0\right\}$. Since $N_{\rho, 0}\in J_{\rho, 0}$ and $N_{\rho, k}\in J_{\rho, k}$, we have 

$$x_{N_{\rho, 0}, N_{\rho, 0}}^\rho=\tfrac{N_{\rho, 0}}{2N_{\rho, 0}+1} \text{ and }x_{N_{\rho, k}, N_{\rho, k}}^\rho = \tfrac{N_{\rho, k}-k}{2N_{\rho, k}+1}$$

Therefore 

$$\begin{aligned}
&x_{N_{\rho, 0}, N_{\rho, 0}}^\rho\geqslant x_{N_{\rho, k}, N_{\rho, k}}^\rho \\
\iff&(2N_{\rho, 0}+1)(N_{\rho, k}-k)\leqslant(2N_{\rho, k}+1)N_{\rho, 0} \\
\iff& N_{\rho, k}\leqslant k+N_{\rho, 0}(2k+1)
\end{aligned}$$

which is true because $N_{\rho, 0}\geqslant1$ and according to Lemma \ref{lem:bounds on maxima, >0.5}. This proves the theorem.

\end{proof}

Note that 

\begin{proposition}\label{prop:best diagonal point>rho/2, >0.5}
One has 

$$\tfrac{2}{3}\rho>\tfrac{\left\lfloor\tfrac{1}{1-\rho}\right\rfloor-1}{2\left\lfloor\tfrac{1}{1-\rho}\right\rfloor-1}>\tfrac{1}{2}\rho$$
\end{proposition}

\begin{proof}
Let $k\in\mathbb{N}\setminus\left\{0, 1\right\}$. One has 

$$\left\lfloor\tfrac{1}{1-\rho}\right\rfloor = k \iff \rho\in\left(1-\tfrac{1}{k}, 1-\tfrac{1}{k+1}\right)$$

Then if $\rho\in\left(1-\tfrac{1}{k}, 1-\tfrac{1}{k+1}\right)$, one has 

$$\tfrac{\left\lfloor\tfrac{1}{1-\rho}\right\rfloor-1}{2\left\lfloor\tfrac{1}{1-\rho}\right\rfloor-1} = \tfrac{k-1}{2k-1}\geqslant \tfrac{1}{2}\left(1-\tfrac{1}{k+1}\right)>\tfrac{1}{2}\rho$$

and 

$$\tfrac{k-1}{2k-1}\leqslant\tfrac{2}{3}\left(1-\tfrac{1}{k}\right)<\tfrac{2}{3}\rho$$
\end{proof}

Note that the factors $\tfrac{1}{2}$ and $\tfrac{2}{3}$ are optimal, since they can be approached arbitrarily close as $\rho$ goes to $1$ from below (respectively to $\tfrac{1}{2}$ from above). Also note that we have the simpler upper bound 

$$\tfrac{\left\lfloor\tfrac{1}{1-\rho}\right\rfloor-1}{2\left\lfloor\tfrac{1}{1-\rho}\right\rfloor-1}<\tfrac{1}{2}$$

\begin{figure}[!h]
\begin{center}
\includegraphics[width=11cm]{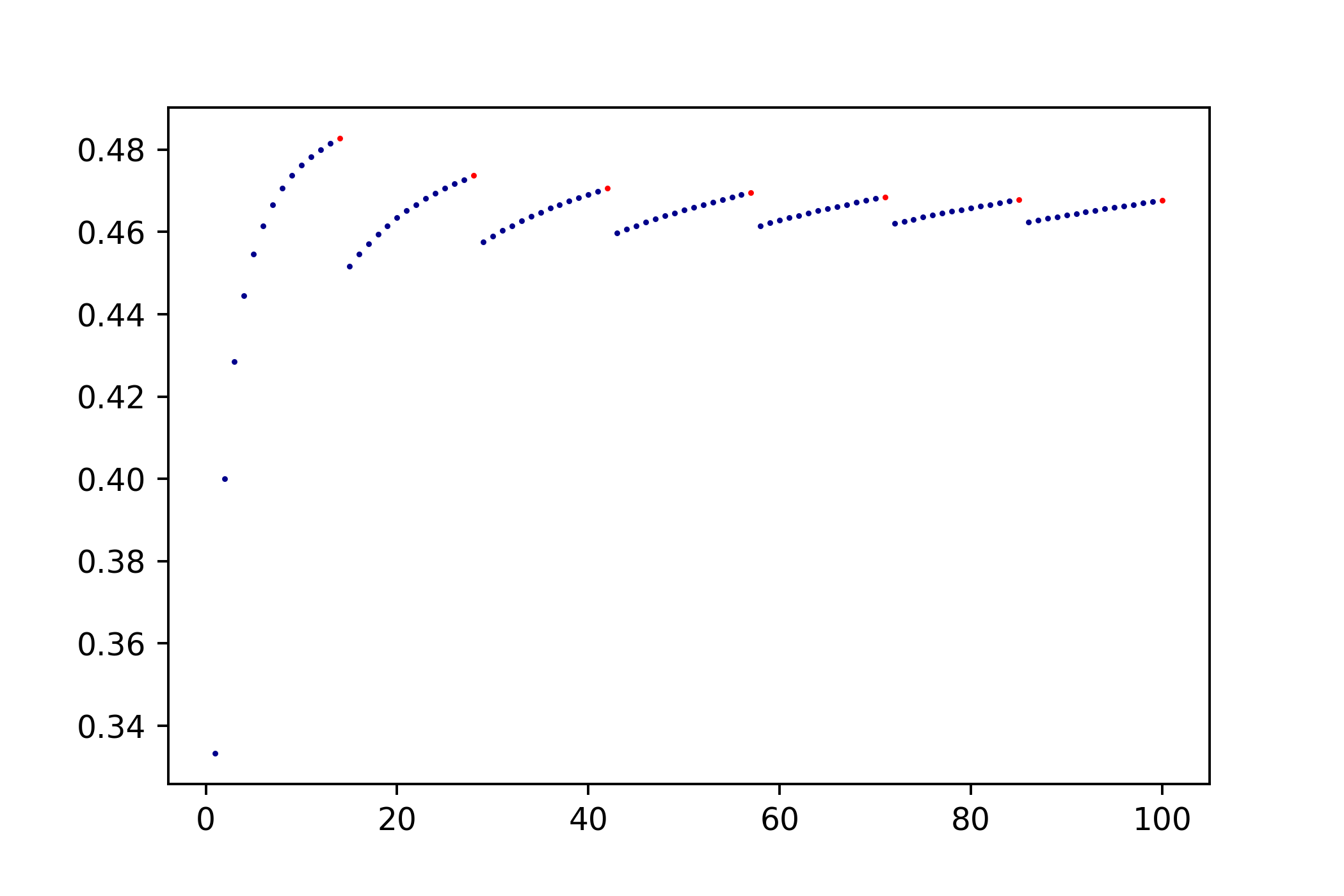}
\end{center}
    \vskip-1cm\caption{First 100 terms of sequence $(x_{n, n}^\rho)_{n \in \mathbb{N}\setminus\left\{0\right\}}$ for $\rho=0.93+\pi.10^{-5}$. Red points are such that $n\notin I_\rho$ and blue ones are such that $n\in I_\rho$.}
    %\label{fig:Fig. 7}
\end{figure}

\subsection{Case $\rho<\frac{1}{2}$}

In this subsection, we assume that $\rho<\tfrac{1}{2}$. By Proposition \ref{prop:partition, <0.5}, we have 

$$\forall i\in\mathbb{N}\setminus\left\{0\right\}, \forall k\in Y_{\rho, i}\,,\, x_{N_{\rho, k}', N_{\rho, k}'}^\rho = \tfrac{k}{2k\left\lfloor\tfrac{1}{\rho}\right\rfloor+2i-1}$$

Hence we obtain the following

\begin{corollary}\label{coro:x_nn increasing on new partition, <0.5}
Let $i\in\mathbb{N}\setminus\left\{0\right\}$. The sequence $\left(x_{N_{\rho, k}', N_{\rho, k}'}^\rho\right)_{k\in \mathbb{N}\setminus\left\{0\right\}}$ is strictly increasing on the interval~$Y_{\rho, i}$.
\end{corollary}

\begin{lemma}\label{lem:bounds on maxima, <0.5}
For any $j\in\mathbb{N}\setminus\left\{0, 1\right\}$, one has 

$$j M_{\rho, 1}\leqslant M_{\rho, j}\leqslant j M_{\rho, 1}+j-1$$
\end{lemma}

\begin{proof}
We prove following claim 

$$\forall j\in\mathbb{N}\setminus\left\{0\right\}, M_{\rho, 1}\leqslant M_{\rho, j+1}-M_{\rho, j}\leqslant M_{\rho, 1}+1$$

from which the result follows. Let $j\in\mathbb{N}\setminus\left\{0\right\}$. Then we have

$$\begin{aligned}
M_{\rho, j+1}-M_{\rho, j} &= \left\lfloor\tfrac{j\rho}{1-\left\lfloor\tfrac{1}{\rho}\right\rfloor\rho}+\tfrac{\rho}{1-\left\lfloor\tfrac{1}{\rho}\right\rfloor\rho}\right\rfloor-\left\lfloor\tfrac{j\rho}{1-\left\lfloor\tfrac{1}{\rho}\right\rfloor\rho}\right\rfloor \\
&\geqslant \left\lfloor\tfrac{j\rho}{1-\left\lfloor\tfrac{1}{\rho}\right\rfloor\rho}+\left\lfloor\tfrac{\rho}{1-\left\lfloor\tfrac{1}{\rho}\right\rfloor\rho}\right\rfloor\right\rfloor-\left\lfloor\tfrac{j\rho}{1-\left\lfloor\tfrac{1}{\rho}\right\rfloor\rho}\right\rfloor \\
&=M_{\rho, 1}
\end{aligned}$$

and in the same way one obtains 

$$M_{\rho, j+1}-M_{\rho, j}\leqslant M_{\rho, 1}+1$$

Thus the claim we stated at the beginning of the proof holds.
\end{proof}

We conclude with the following

\begin{theorem}[Best diagonal point, $\rho<\tfrac{1}{2}$]\label{thm:best diagonal point, <0.5}
If $\rho<\tfrac{1}{2}$, then 

$$\begin{aligned}\underset{n\in I_\rho}\max\, x_{n, n}^\rho &= x_{N_{\rho, M_{\rho, 1}}', N_{\rho, M_{\rho, 1}}'}^\rho\\
&= \tfrac{M_{\rho, 1}}{2N_{\rho, M_{\rho, 1}}'+1}\\
&= \tfrac{\left\lfloor\tfrac{\rho}{1-\left\lfloor\tfrac{1}{\rho}\right\rfloor\rho}\right\rfloor}{2\left\lfloor\tfrac{\rho}{1-\left\lfloor\tfrac{1}{\rho}\right\rfloor\rho}\right\rfloor\left\lfloor\tfrac{1}{\rho}\right\rfloor+1}
\end{aligned}$$

with 

$$N_{\rho, M_{\rho, 1}}' \overset{\text{def}}=\left\lfloor\tfrac{M_{\rho, 1}}{\rho}\right\rfloor=M_{\rho, 1}\left\lfloor\tfrac{1}{\rho}\right\rfloor \text{ and }M_{\rho, 1}=\left\lfloor\tfrac{\rho}{1-\left\lfloor\tfrac{1}{\rho}\right\rfloor\rho}\right\rfloor$$
\end{theorem}

\begin{proof}
We start with the problem of finding $\underset{n\in I_\rho}\sup\, x_{n, n}^\rho$. Thanks to Proposition \ref{prop:partition, <0.5}, the problem can be solved by finding $\underset{k\in\mathbb{N}\setminus\left\{0\right\}}\sup\, x_{N_{\rho, k}', N_{\rho, k}'}^\rho$. Then, thanks to Corollary \ref{coro:x_nn increasing on new partition, <0.5}, the problem can be reduced to finding $\underset{j\in\mathbb{N}\setminus\left\{0\right\}}\sup\, x_{N_{\rho, M_{\rho, j}}', N_{\rho, M_{\rho, j}}'}^\rho$. Finally, let's prove that this supremum is achieved and $$\underset{j\in\mathbb{N}\setminus\left\{0\right\}}\max\, x_{N_{\rho, M_{\rho, j}}', N_{\rho, M_{\rho, j}}'}^\rho = x_{N_{\rho, M_{\rho, 1}}', N_{\rho, M_{\rho, 1}}'}^\rho.$$ 

Let $j\in\mathbb{N}\setminus\left\{0, 1\right\}$. Since $N_{\rho, M_{\rho, 1}}'\in J_{\rho, M_{\rho, 1}}'$ and $N_{\rho, M_{\rho, j}}'\in J_{\rho, M_{\rho, j}}'$, we have 

$$x_{N_{\rho, M_{\rho, 1}}', N_{\rho, M_{\rho, 1}}'}^\rho=\tfrac{M_{\rho, 1}}{2N_{\rho, M_{\rho, 1}}'+1} \text{ and }x_{N_{\rho, M_{\rho, j}}', N_{\rho, M_{\rho, j}}'}^\rho = \tfrac{M_{\rho, j}}{2N_{\rho, M_{\rho, j}}'+1}$$

Note also, that according to the proof of Proposition \ref{prop:partition, <0.5}, since $M_{\rho, 1}\in Y_{\rho, 1}$ and $M_{\rho, j}\in Y_{\rho, j}$, we have

$$N_{\rho, M_{\rho, 1}}'=M_{\rho, 1}\left\lfloor\tfrac{1}{\rho}\right\rfloor \text{ and }N_{\rho, M_{\rho, j}}'=M_{\rho, j}\left\lfloor\tfrac{1}{\rho}\right\rfloor+j-1$$

Therefore 

$$\begin{aligned}
&x_{N_{\rho, M_{\rho, 1}}', N_{\rho, M_{\rho, 1}}'}^\rho\geqslant x_{N_{\rho, M_{\rho, j}}', N_{\rho, M_{\rho, j}}'}^\rho \\
\iff& M_{\rho, j}\left(2M_{\rho, 1}\left\lfloor\tfrac{1}{\rho}\right\rfloor+1\right)\leqslant M_{\rho, 1}\left(2M_{\rho, j}\left\lfloor\tfrac{1}{\rho}\right\rfloor+2j-1\right) \\
\iff& M_{\rho, j}\leqslant(2j-1)M_{\rho, 1}
\end{aligned}$$

which is true because $M_{\rho, 1}\geqslant1$ and according to Lemma \ref{lem:bounds on maxima, <0.5}. This proves the theorem.

\end{proof}

This leads to the following.

\begin{proposition}\label{prop:best diagonal point>rho/2, <0.5}
One has 

$$\tfrac{3}{5}\rho>\tfrac{\left\lfloor\tfrac{\rho}{1-\left\lfloor\tfrac{1}{\rho}\right\rfloor\rho}\right\rfloor}{2\left\lfloor\tfrac{\rho}{1-\left\lfloor\tfrac{1}{\rho}\right\rfloor\rho}\right\rfloor\left\lfloor\tfrac{1}{\rho}\right\rfloor+1}>\tfrac{1}{2}\rho$$
\end{proposition}

\begin{proof}
Let $k\in\mathbb{N}\setminus\left\{0, 1\right\}$ and $l\in\mathbb{N}\setminus\left\{0\right\}$. One has

$$\left\lfloor\tfrac{1}{\rho}\right\rfloor = k \iff \rho\in\left(\tfrac{1}{k+1}, \tfrac{1}{k}\right)$$

and if this condition is fulfilled, one has

$$\left\lfloor\tfrac{\rho}{1-k\rho}\right\rfloor=l \iff \rho\in\left(\tfrac{l}{lk+1}, \tfrac{l+1}{(l+1)k+1}\right)$$

Then if $\rho\in\left(\tfrac{1}{k+1}, \tfrac{1}{k}\right)\cap\left(\tfrac{l}{lk+1}, \tfrac{l+1}{(l+1)k+1}\right)$, one has 

$$\tfrac{\left\lfloor\tfrac{\rho}{1-\left\lfloor\tfrac{1}{\rho}\right\rfloor\rho}\right\rfloor}{2\left\lfloor\tfrac{\rho}{1-\left\lfloor\tfrac{1}{\rho}\right\rfloor\rho}\right\rfloor\left\lfloor\tfrac{1}{\rho}\right\rfloor+1} = \tfrac{l}{2kl+1} \geqslant \tfrac{1}{2}\tfrac{l+1}{(l+1)k+1}>\tfrac{1}{2}\rho$$

and 

$$\tfrac{l}{2kl+1} \leqslant \tfrac{3}{5}\tfrac{l}{lk+1}<\tfrac{3}{5}\rho$$
\end{proof}

Note that the factors $\tfrac{1}{2}$ and $\tfrac{3}{5}$ are optimal, since they can be approached arbitrarily close as $\rho$ goes to $\tfrac{1}{2}$ from below (respectively to $\tfrac{1}{3}$ from above).

\begin{figure}[!h]
\begin{center}
\includegraphics[width=11cm]{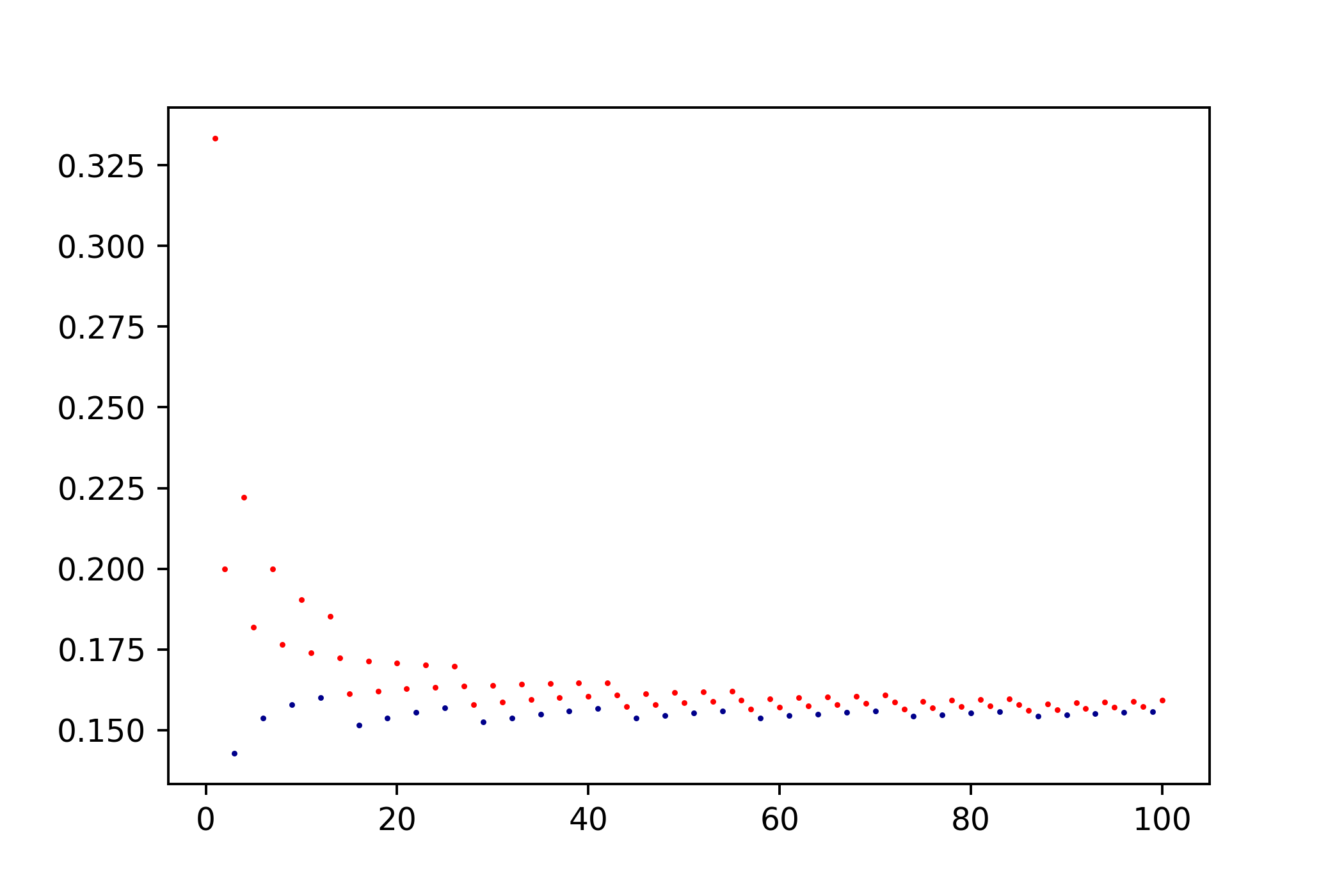}
\end{center}
    \vskip-1cm\caption{First 100 terms of sequence $(x_{n, n}^\rho)_{n \in \mathbb{N}\setminus\left\{0\right\}}$ for $\rho=0.31+\pi.10^{-5}$. Red points are such that $n\notin I_\rho$ and blue ones are such that $n\in I_\rho$.}
    %\label{fig:Fig. 7}
\end{figure}

\section{Is the best diagonal point an extreme point ?}\label{sec:3}

\begin{definition}\label{def:best diagonal point}
We denote 

$$D_\rho:=(d_\rho, d_\rho),$$

the \textit{best diagonal point} obtained in Theorems \ref{thm:best diagonal point, >0.5} and \ref{thm:best diagonal point, <0.5}.
\end{definition}

The next step consists in determining whether this \textit{best diagonal point} is an extreme point of $\Lambda_\rho$ or not (we already know that it is the only candidate left among points on the diagonal). For some values of $\rho$, in order to prove that $D_\rho$ is an extreme point of $\Lambda_\rho$, we will prove that all points of $L_\rho$ except $D_\rho$ lie strictly under straight line $x+y=2d_\rho$ ($D_\rho$ belongs to that line). For other values of $\rho$, we will prove that some non-diagonal points of $L_\rho$ belong to the line $x+y=2d_\rho$. Therefore in this case, point $D_\rho$ is not an extreme point of $\Lambda_\rho$.

\subsection{Case $\rho>\frac{1}{2}$}

In this subsection, we assume that $\rho>\tfrac{1}{2}$. Let's remind the reader about the following formulas, coming from Proposition \ref{prop:partition, >0.5}. For all $k, l\in\mathbb{N}$ and $m\in I_{\rho, k}, n\in I_{\rho, l}$ we have 

$$\begin{aligned}
x_{m, n}^\rho &= \tfrac{m-k}{m+n+1} \\
y_{m, n}^\rho &= \tfrac{n-l}{m+n+1} \\
\text{hence }x_{m, n}^\rho+y_{m, n}^\rho &= \tfrac{m+n-k-l}{m+n+1}
\end{aligned}$$

Therefore we have the following

\begin{lemma}\label{lem:x_mn+y_mn increasing, >0.5}
For any $k, l\in\mathbb{N}$, sequence $\left(x_{m, n}^\rho+y_{m, n}^\rho\right)_{(m, n)\in I_\rho\times I_\rho}$ is strictly increasing with respect to its first (respectively second) index on the tile $I_{\rho, k}\times I_{\rho, l}$. 
\end{lemma}

\begin{corollary}\label{coro:maxima x_mn+y_mn, >0.5}
Let $k, l\in\mathbb{N}$. Then 

$$\underset{\substack{m\in I_{\rho, k} \\ n\in I_{\rho, l} \\ m\neq n}}\max \left(x_{m, n}^\rho+y_{m, n}^\rho\right) = \left\{
\begin{aligned}
\tfrac{2N_{\rho, k}-1-2k}{2N_{\rho, k}} &\text{ if }k=l \\
\tfrac{N_{\rho, k}+N_{\rho, l}-k-l}{N_{\rho, k}+N_{\rho, l}+1} &\text{ if }k\neq l
\end{aligned}\right.$$

\noindent More precisely, the maximum is attained respectively only for couples of indices \newline $\left(N_{\rho, k}, N_{\rho, k}-1\right), \left(N_{\rho, k}-1, N_{\rho, k}\right)$ in the first case and $\left(N_{\rho, k}, N_{\rho, l}\right)$ in the second case.
\end{corollary}

\begin{proposition}\label{prop:x+y<2d, >2/3}
Suppose $\rho>\tfrac{2}{3}$. For any $m, n\in I_\rho$ such that $m\neq n$, we have 

$$x_{m, n}^\rho+y_{m, n}^\rho<2d_\rho$$
\end{proposition}

\begin{proof}
Let $k, l\in\mathbb{N}$. It is enough to show 

$$\underset{\substack{m\in I_{\rho, k} \\ n\in I_{\rho, l} \\ m\neq n}}\max \left(x_{m, n}^\rho+y_{m, n}^\rho\right)<2d_\rho = \tfrac{2N_{\rho, 0}}{2N_{\rho, 0}+1}$$

\begin{itemize}

\item If $k\neq l$, according to Corollary \ref{coro:maxima x_mn+y_mn, >0.5} we have 

$$\begin{aligned}
&\underset{\substack{m\in I_{\rho, k} \\ n\in I_{\rho, l} \\ m\neq n}}\max \left(x_{m, n}^\rho+y_{m, n}^\rho\right)<2d_\rho \\
\iff& \left(2N_{\rho, 0}+1\right)\left(N_{\rho, k}+N_{\rho, l}-k-l\right)<2N_{\rho, 0}\left(N_{\rho, k}+N_{\rho, l}+1\right) \\
\iff& N_{\rho, k}+N_{\rho, l}<(k+l)\left(2N_{\rho, 0}+1\right)+2N_{\rho, 0}
\end{aligned}$$

But according to Lemma \ref{lem:bounds on maxima, >0.5}, we have 

$$\begin{aligned}
N_{\rho, k}+N_{\rho, l}&\leqslant (k+l)\left(N_{\rho, 0}+2\right)+2N_{\rho, 0} \\
&<(k+l)\left(2N_{\rho, 0}+1\right)+2N_{\rho, 0} \text{ because }N_{\rho, 0}\geqslant 2 \text{ since }\rho>\tfrac{2}{3} \\
\end{aligned}$$

\item If $k=l$, according to Corollary \ref{coro:maxima x_mn+y_mn, >0.5} we have

$$\begin{aligned}
&\underset{\substack{m\in I_{\rho, k} \\ n\in I_{\rho, l} \\ m\neq n}}\max \left(x_{m, n}^\rho+y_{m, n}^\rho\right)<2d_\rho \\
\iff& \left(2N_{\rho, k}-1-2k\right)\left(2N_{\rho, 0}+1\right)<4N_{\rho, 0}N_{\rho, k} \\
\iff& 2N_{\rho, k}<(1+2k)\left(2N_{\rho, 0}+1\right)
\end{aligned}$$

But according to Lemma \ref{lem:bounds on maxima, >0.5}, we have 

$$\begin{aligned}
2N_{\rho, k}&\leqslant 2k\left(N_{\rho, 0}+2\right)+2N_{\rho, 0} \\
&<(1+2k)\left(2N_{\rho, 0}+1\right) \\
\end{aligned}$$

\end{itemize}

Therefore, the initial claim is true.

\end{proof}

From the fact that for any $k\in\mathbb{N}$, we have 

$$N_{\rho, 0}+1\leqslant N_{\rho, k+1}-N_{\rho, k}\leqslant N_{\rho, 0}+2$$

we obtain the following 

\begin{lemma}\label{lem:equality case in bounds on maxima lemma, >0.5}
Suppose $\rho<\tfrac{2}{3}$. Integer sequence 

$$\begin{aligned}
u_{\rho, k}&:=k\left(N_{\rho, 0}+2\right)+N_{\rho, 0}-N_{\rho, k} \\
&\phantom{:}=3k+1-N_{\rho, k} \text{ since }N_{\rho, 0}=1
\end{aligned}$$ 

is increasing, with $u_{\rho, 0}=0$. Moreover we have

$$\max\left\{k\in\mathbb{N}:u_{\rho, k}=0\right\}=\left\lfloor\tfrac{2\rho-1}{2-3\rho}\right\rfloor$$
\end{lemma}

Note that 

$$\left\lfloor\tfrac{2\rho-1}{2-3\rho}\right\rfloor=0\iff\rho<\tfrac{3}{5}$$

\begin{proof}
Let's prove the second claim. For any $k\in\mathbb{N}$, we have 

$$\begin{aligned}
&u_{\rho, k}=0 \\
\iff& \left\lfloor\tfrac{k+1}{1-\rho}\right\rfloor=3k+2 \\
\iff& 3k+2<\tfrac{k+1}{1-\rho}<3k+3 \\
\iff& -1<k<\tfrac{2\rho-1}{2-3\rho}
\end{aligned}$$
\end{proof}

The sequence $\left(u_{\rho, k}\right)_{k\in\mathbb{N}}$ appears naturally because we will need to know whether or not equality can occur in the right inequality of Lemma \ref{lem:bounds on maxima, >0.5}.

\begin{proposition}\label{prop:x+y<2d, 2/3>rho>0.5}
Suppose $\rho<\tfrac{2}{3}$. If $\rho<\tfrac{3}{5}$, then for any $m, n\in I_\rho$ such that $m\neq n$, we have 

$$x_{m, n}^\rho+y_{m, n}^\rho<2d_\rho$$

Otherwise, if $\rho\in\left(\tfrac{3}{5}, \tfrac{2}{3}\right)$, the large inequality holds true, and equality case occurs precisely for indices $(m, n)=\left(N_{\rho, k}, N_{\rho, l}\right)$ with distinct $k, l\in\left\{0, 1, \hdots, \left\lfloor\tfrac{2\rho-1}{2-3\rho}\right\rfloor\right\}$.
\end{proposition}

\begin{proof}
Let $k, l\in\mathbb{N}$. Once again, it is enough to show 

$$\underset{\substack{m\in I_{\rho, k} \\ n\in I_{\rho, l} \\ m\neq n}}\max \left(x_{m, n}^\rho+y_{m, n}^\rho\right)<2d_\rho = \tfrac{2N_{\rho, 0}}{2N_{\rho, 0}+1}$$

The case when $k=l$ works exactly the same way as in the proof of Proposition~\ref{prop:x+y<2d, >2/3}. If $k\neq l$, then we had the condition

$$\begin{aligned}
&\underset{\substack{m\in I_{\rho, k} \\ n\in I_{\rho, l} \\ m\neq n}}\max \left(x_{m, n}^\rho+y_{m, n}^\rho\right)<2d_\rho \\
\iff& N_{\rho, k}+N_{\rho, l}<(k+l)\left(2N_{\rho, 0}+1\right)+2N_{\rho, 0} \\
\iff& N_{\rho, k}+N_{\rho, l}<3(k+l)+2 \text{ since }N_{\rho, 0}=1 \\
\iff& N_{\rho, k}<3k+1 \text{ or }N_{\rho, l}<3l+1 \text{ according to Lemma \ref{lem:bounds on maxima, >0.5}} \\
\iff& u_{\rho, k}>0 \text{ or }u_{\rho, l}>0 
\end{aligned}$$

Therefore according to Lemma \ref{lem:equality case in bounds on maxima lemma, >0.5}

\begin{itemize}
\item if $\rho<\tfrac{3}{5}$, this condition is fulfilled for any distinct $k, l\in\mathbb{N}$, since $\left\lfloor\tfrac{2\rho-1}{2-3\rho}\right\rfloor=0$.
\item if $\rho\in\left(\tfrac{3}{5}, \tfrac{2}{3}\right)$, then this condition is not fulfilled (meaning that we have an equality instead of strict inequality in the initial statement) precisely for distinct \\ $k, l\in\left\{0, 1, \hdots, \left\lfloor\tfrac{2\rho-1}{2-3\rho}\right\rfloor\right\}$ since $\left\lfloor\tfrac{2\rho-1}{2-3\rho}\right\rfloor\geqslant1$.
\end{itemize}
\end{proof} 

Note that if $\rho\in\left(\tfrac{1}{2}, \tfrac{2}{3}\right)$, we have 

$$\tfrac{1}{1-\rho}\in(2, 3) \text{ and therefore }N_{\rho, 0}=1$$

We will denote $k_\rho = \left\lfloor\tfrac{2\rho-1}{2-3\rho}\right\rfloor$ in the following theorem.

\begin{theorem}[When the best diagonal point is an extreme point]\label{thm:is the best diagonal point an extreme point ?, >0.5}
\item If $\rho\in\left(\tfrac{3}{5}, \tfrac{2}{3}\right)$, then $D_\rho$ is not an extreme point of $\Lambda_\rho$; instead it is the middle point of points $A^\rho_{N_{\rho, 0}, N_{\rho, k_\rho}}$ and $A^\rho_{N_{\rho, k_\rho}, N_{\rho, 0}}$, which belong to line $x+y=2d_\rho$ and are extreme points of $\Lambda_\rho$. If $\rho\in\left(\tfrac{1}{2}, \tfrac{3}{5}\right)\cup\left(\tfrac{2}{3}, 1\right)$, then $D_\rho$ is an extreme point of $\Lambda_\rho$.
\end{theorem}

\begin{proof}
\begin{itemize}
\item Suppose $\rho\in\left(\tfrac{1}{2}, \tfrac{3}{5}\right)\cup\left(\tfrac{2}{3}, 1\right)$. Then according to Propositions \ref{prop:x+y<2d, >2/3} and \ref{prop:x+y<2d, 2/3>rho>0.5}, for all distinct $m, n\in I_\rho$, we have

$$x_{m, n}^\rho+y_{m, n}^\rho<2d_\rho$$

and hence $D_\rho$ is an extreme point of $\Lambda_\rho$.

\item Suppose $\rho\in\left(\tfrac{3}{5}, \tfrac{2}{3}\right)$.

\begin{figure}[!h]
\begin{center}
\includegraphics[width=11cm]{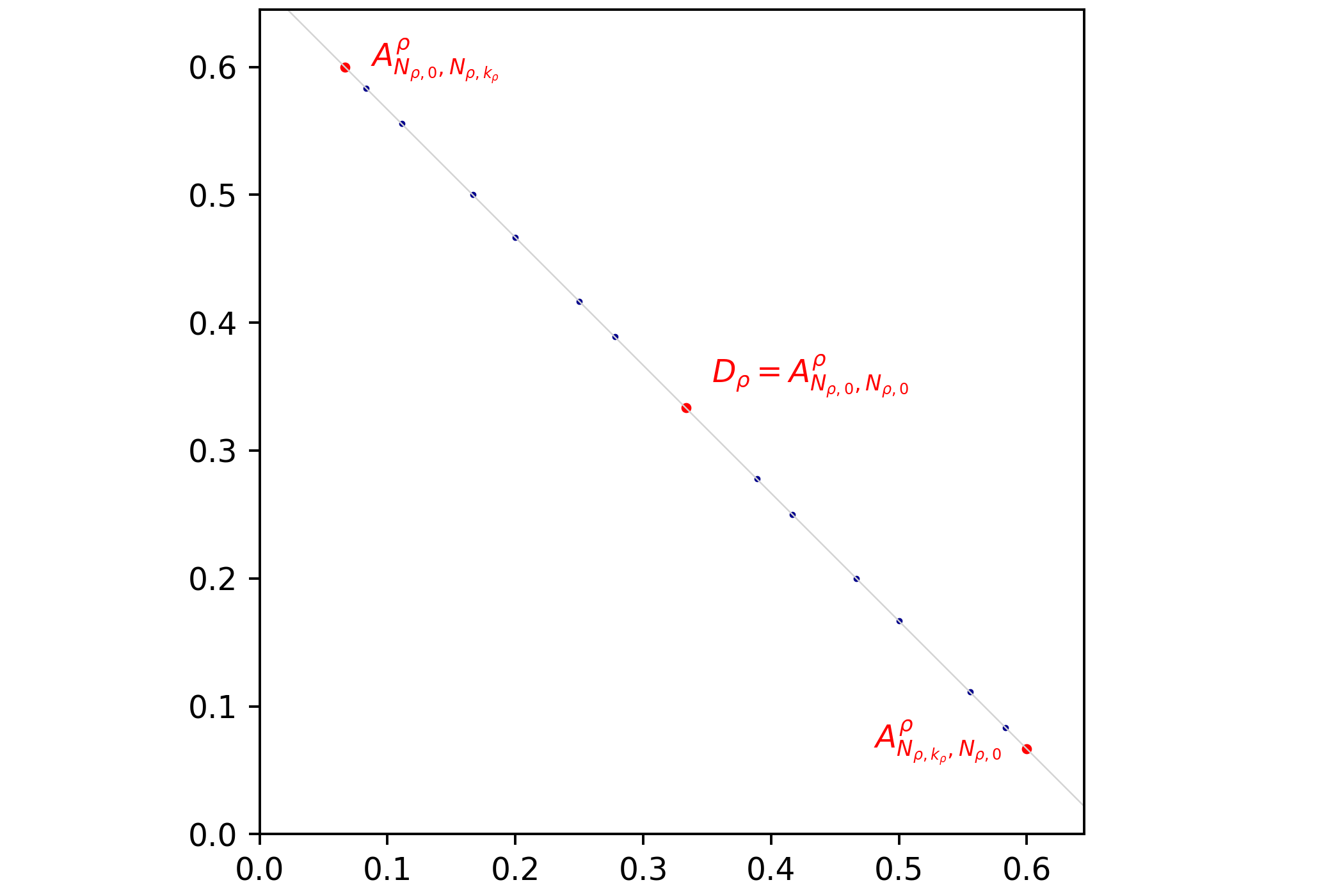}
\end{center}
    \caption{Points $A_{N_{\rho, k}, N_{\rho, l}}^\rho$ for distinct $k, l\in\left\{0, 1, \hdots, k_\rho\right\}$ for $\rho=0.645+\pi.10^{-5}$. Straight line $y=-x+2d_\rho$ is in gray.}
    %\label{fig:Fig. 7}
\end{figure}

According to the proof of Proposition \ref{prop:x+y<2d, 2/3>rho>0.5}, we have for any distinct \\ $k, l\in\left\{0, 1, \hdots, k_\rho\right\}$

$$x_{N_{\rho, k}, N_{\rho, l}}^\rho+y_{N_{\rho, 0}, N_{\rho, 1}}^\rho = 2d_\rho$$

which means that such points $A_{N_{\rho, k}, N_{\rho, l}}^\rho$ (and $D_\rho$) lie on line $x+y=2d_\rho$ (and those are the only points of $L_\rho$ lying on this line). Additionally, for any such indices $k, l$, we have 

$$\begin{aligned}
x_{N_{\rho, k}, N_{\rho, l}}^\rho &= \tfrac{N_{\rho, k}-k}{N_{\rho, k}+N_{\rho, l}+1} \\
&= \tfrac{2k+1}{3(k+l+1)}
\end{aligned}$$

because $u_{\rho, k}=u_{\rho, l}=0$ (see Lemma \ref{lem:equality case in bounds on maxima lemma, >0.5}) and $N_{\rho, 0}=1$. This sequence is strictly decreasing with respect to $l$ and strictly increasing with respect to $k$. Therefore $A_{N_{\rho, 0}, N_{\rho, k_\rho}}^\rho$ and $A_{N_{\rho, k_\rho}, N_{\rho, 0}}^\rho$ have respectively minimal and maximal $x$ coordinate among points of $L_\rho$ lying on line $y=-x+2d_\rho$. As a consequence,  $A_{N_{\rho, 0}, N_{\rho, k_\rho}}^\rho$ and $A_{N_{\rho, k_\rho}, N_{\rho, 0}}^\rho$ are extreme points of $\Lambda_\rho$. \\

Finally, by symmetry, the middle point of points $A_{N_{\rho, 0}, N_{\rho, k_\rho}}^\rho, A_{N_{\rho, k_\rho}, N_{\rho, 0}}^\rho$ is on the diagonal and therefore it is precisely $D_\rho$ (which shows that $D_\rho$ is not an extreme point of $\Lambda_\rho$).
\end{itemize}
\end{proof}

\subsection{Case $\rho<\frac{1}{2}$}

In this subsection, we assume that $\rho<\tfrac{1}{2}$. Let's remind the reader about the following formulas, coming from Proposition \ref{prop:partition, <0.5}. One has 

$$I_\rho = \left\{N_{\rho, k}'\,|\,k\in\mathbb{N}\setminus\left\{0\right\}\right\}$$ 

and for all $i, j\in\mathbb{N}\setminus\left\{0\right\}$ and $k\in Y_{\rho, i}, l\in Y_{\rho, j}$ we have 

$$\begin{aligned}
x_{N_{\rho, k}', N_{\rho, l}'}^\rho &= \tfrac{k}{(k+l)\left\lfloor\tfrac{1}{\rho}\right\rfloor+i+j-1} \\
y_{N_{\rho, k}', N_{\rho, l}'}^\rho &= \tfrac{l}{(k+l)\left\lfloor\tfrac{1}{\rho}\right\rfloor+i+j-1} \\
\text{hence }x_{N_{\rho, k}', N_{\rho, l}'}^\rho+y_{N_{\rho, k}', N_{\rho, l}'}^\rho &= \tfrac{k+l}{(k+l)\left\lfloor\tfrac{1}{\rho}\right\rfloor+i+j-1}
\end{aligned}$$

Therefore we have the following  

\begin{lemma}\label{lem:x_mn+ymn increasing, <0.5}
For any $i, j\in\mathbb{N}\setminus\left\{0\right\}$, sequence $\left(x_{N_{\rho, k}', N_{\rho, l}'}^\rho+y_{N_{\rho, k}', N_{\rho, l}'}^\rho\right)_{(k, l)\in \mathbb{N}\setminus\left\{0\right\}\times \mathbb{N}\setminus\left\{0\right\}}$ is strictly increasing with respect to its first (respectively second) index on the tile $Y_{\rho, i}\times Y_{\rho, j}$. 
\end{lemma}

\begin{corollary}\label{coro:maxima x_mn+y_mn, <0.5}
Let $i, j\in\mathbb{N}\setminus\left\{0\right\}$. Then 

$$\underset{\substack{k\in Y_{\rho, i} \\ l\in Y_{\rho, j} \\ k\neq l}}\max \left(x_{N_{\rho, k}', N_{\rho, l}'}^\rho+y_{N_{\rho, k}', N_{\rho, l}'}^\rho\right) = \left\{
\begin{aligned}
\tfrac{2M_{\rho, i}-1}{\left(2M_{\rho, i}-1\right)\left\lfloor\tfrac{1}{\rho}\right\rfloor+2i-1} &\text{ if }i=j \text{ and Card}(Y_{\rho, i})\geqslant2 \\
\tfrac{M_{\rho, i}+M_{\rho, j}}{\left(M_{\rho, i}+M_{\rho, j}\right)\left\lfloor\tfrac{1}{\rho}\right\rfloor+i+j-1} &\text{ if }i\neq j
\end{aligned}\right.$$

\noindent More precisely, the maximum is attained respectively only for couples of indices 
\newline $\left(M_{\rho, i}, M_{\rho, i}-1\right), \left(M_{\rho, i}-1, M_{\rho, i}\right)$ in the first case and $\left(M_{\rho, i}, M_{\rho, j}\right)$ in the second case.
\end{corollary}

Note that according to Lemma \ref{lem:explicit expression new partition, <0.5}, for $i\geqslant2$ the condition $\text{Card}(Y_{\rho, i})\geqslant2$ is equivalent to $M_{\rho, i}-M_{\rho, i-1}\geqslant2$ (and $M_{\rho, 1}-1\geqslant2$ when $i=1$), and it is not always satisfied, depending on $\rho$.

\begin{lemma}\label{lem:M_rho1=1, <0.5}
One has 

$$M_{\rho, 1}=1\iff\rho\in\bigcup_{t=2}^\infty\left(\tfrac{1}{t+1}, \tfrac{2}{2t+1}\right)$$
\end{lemma}

\begin{definition}\label{def:values of rho such that M_rho1=1, <0.5}
Let's denote 

$$E = \left(\mathbb{R}\setminus\mathbb{Q}\right)\cap\left(\bigcup_{t=2}^\infty\left(\tfrac{1}{t+1}, \tfrac{2}{2t+1}\right)\right)$$

and for any $t\in\mathbb{N}\setminus\left\{0, 1\right\}$,

$$E_{t}=\left(\mathbb{R}\setminus\mathbb{Q}\right)\cap\left(\tfrac{1}{t+1}, \tfrac{2}{2t+1}\right)$$
\end{definition}

\begin{proposition}\label{prop:x+y<2d, rho not in E}
Suppose $\rho\notin E$. For any $m, n\in I_\rho$ such that $m\neq n$, we have 

$$x_{m, n}^\rho+y_{m, n}^\rho<2d_\rho$$
\end{proposition}

\begin{proof}
Let $i, j\in\mathbb{N}\setminus\left\{0\right\}$. It is enough to show 

$$\underset{\substack{k\in Y_{\rho, i} \\ l\in Y_{\rho, j} \\ k\neq l}}\max \left(x_{N_{\rho, k}', N_{\rho, l}'}^\rho+y_{N_{\rho, k}', N_{\rho, l}'}^\rho\right)<2d_\rho = \tfrac{2M_{\rho, 1}}{2\left\lfloor\tfrac{1}{\rho}\right\rfloor M_{\rho, 1}+1}$$

\begin{itemize}

\item If $i\neq j$, according to Corollary \ref{coro:maxima x_mn+y_mn, <0.5} we have 

$$\begin{aligned}
&\underset{\substack{k\in Y_{\rho, i} \\ l\in Y_{\rho, j} \\ k\neq l}}\max \left(x_{N_{\rho, k}', N_{\rho, l}'}^\rho+y_{N_{\rho, k}', N_{\rho, l}'}^\rho\right)<2d_\rho \\
\iff& \left(2\left\lfloor\tfrac{1}{\rho}\right\rfloor M_{\rho, 1}+1\right)\left(M_{\rho, i}+M_{\rho, j}\right)<2M_{\rho, 1}\left(\left(M_{\rho, i}+M_{\rho, j}\right)\left\lfloor\tfrac{1}{\rho}\right\rfloor+i+j-1\right) \\
\iff& M_{\rho, i}+M_{\rho, j}<2M_{\rho, 1}(i+j-1)
\end{aligned}$$

But according to Lemmas \ref{lem:bounds on maxima, <0.5} and \ref{lem:M_rho1=1, <0.5}, we have 

$$\begin{aligned}
M_{\rho, i}+M_{\rho, j}&\leqslant (i+j)M_{\rho, 1}+i+j-2 \\
&<2M_{\rho, 1}(i+j-1) \text{ because }M_{\rho, 1}\geqslant 2 \text{ since }\rho\notin E \\
\end{aligned}$$

\item If $i=j$ and Card$(Y_{\rho, i})\geqslant2$, according to Corollary \ref{coro:maxima x_mn+y_mn, <0.5} we have

$$\begin{aligned}
&\underset{\substack{k\in Y_{\rho, i} \\ l\in Y_{\rho, j} \\ k\neq l}}\max \left(x_{N_{\rho, k}', N_{\rho, l}'}^\rho+y_{N_{\rho, k}', N_{\rho, l}'}^\rho\right)<2d_\rho \\
\iff& \left(2\left\lfloor\tfrac{1}{\rho}\right\rfloor M_{\rho, 1}+1\right)\left(2M_{\rho, i}-1\right)<2M_{\rho, 1}\left(\left(2M_{\rho, 1}-1\right)\left\lfloor\tfrac{1}{\rho}\right\rfloor+2i-1\right) \\
\iff& M_{\rho, i}<M_{\rho, 1}(2i-1)+\tfrac{1}{2}
\end{aligned}$$

But according to Lemma \ref{lem:bounds on maxima, <0.5}, we have 

$$\begin{aligned}
M_{\rho, i}&\leqslant iM_{\rho, 1}+i-1 \\
&<M_{\rho, 1}(2i-1)+\tfrac{1}{2} \\
\end{aligned}$$

\end{itemize}

Therefore, the initial claim is true.

\end{proof}

From the fact that for any $j\in\mathbb{N}\setminus\left\{0\right\}$, we have 

$$M_{\rho, 1}\leqslant M_{\rho, j+1}-M_{\rho, j}\leqslant M_{\rho, 1}+1$$

and we obtain the following 

\begin{lemma}\label{lem:equality case in bounds on maxima lemma, <0.5}
Suppose $\rho\in E$. The integer sequence 

$$\begin{aligned}
v_{\rho, j}&:=jM_{\rho, 1}+j-1-M_{\rho, j} \\
&\phantom{:}=2j-1-M_{\rho, j} \text{ since }M_{\rho, 1}=1
\end{aligned}$$ 

is increasing, with $v_{\rho, 1}=0$. Moreover, if $\rho\in E_t$ with $t\in\mathbb{N}\setminus\left\{0, 1\right\}$, then we have

$$\max\left\{j\in\mathbb{N}\setminus\left\{0\right\}:v_{\rho, j}=0\right\}=\left\lfloor\tfrac{1-t\rho}{2-(1+2t)\rho}\right\rfloor$$
\end{lemma}

\begin{proof}
Let's prove the second claim. Let $t\in\mathbb{N}\setminus\left\{0, 1\right\}$. For any $j\in\mathbb{N}\setminus\left\{0\right\}$, we have

$$\begin{aligned}
&v_{\rho, j}=0 \\
\iff& \left\lfloor\tfrac{j\rho}{1-\left\lfloor\tfrac{1}{\rho}\right\rfloor\rho}\right\rfloor = 2j-1 \\
\iff& 2j-1<\tfrac{j\rho}{1-\left\lfloor\tfrac{1}{\rho}\right\rfloor\rho}<2j \text{ with }\left\lfloor\tfrac{1}{\rho}\right\rfloor=t \text{ since }\rho\in E_t \\
\iff& -\tfrac{1}{2-(1+2t)\rho}<j<\tfrac{1-t\rho}{2-(1+2t)\rho} \text{ since }2-(1+2t)\rho>0 \text{ because } \rho\in E_t
\end{aligned}$$
\end{proof}

The sequence $\left(v_{\rho, j}\right)_{j\in\mathbb{N}\setminus\left\{0\right\}}$ appears naturally because we will need to know whether or equality can occur in the right inequality of Lemma \ref{lem:bounds on maxima, <0.5}. Note that for any $t\in\mathbb{N}\setminus\left\{0, 1\right\}$, if $\rho\in E_t$, then

$$\left\lfloor\tfrac{1-t\rho}{2-(1+2t)\rho}\right\rfloor>1 \iff \rho>\tfrac{3}{3t+2}$$

Therefore, let's introduce the following notation.

\begin{definition}\label{def:partition of E_t, <0.5}
For any $t\in\mathbb{N}\setminus\left\{0, 1\right\}$, denote

$$\begin{aligned}
F_t &= \left(\mathbb{R}\setminus\mathbb{Q}\right)\cap\left(\tfrac{1}{t+1}, \tfrac{3}{3t+2}\right) \\
G_t &= \left(\mathbb{R}\setminus\mathbb{Q}\right)\cap\left(\tfrac{3}{3t+2}, \tfrac{2}{2t+1}\right) \\
\text{and }F&=\bigcup_{t=2}^\infty F_t
\end{aligned}$$

so that 

$$\begin{aligned}
E_t &= F_t\cup G_t \\
E &= F\cup\left(\bigcup_{t=2}^\infty G_t\right)
\end{aligned}$$
\end{definition}

\begin{proposition}\label{prop:x+y<2d, rho in E}
Suppose $\rho\in E$. If $\rho\in F$, then for any $k, l\in\mathbb{N}\setminus\left\{0\right\}$ such that $k\neq l$, we have 

$$x_{N_{\rho, k}', N_{\rho, l}'}^\rho+y_{N_{\rho, k}', N_{\rho, l}'}^\rho<2d_\rho$$

\noindent Otherwise, if $\rho\in G_t$ with $t\in\mathbb{N}\setminus\left\{0, 1\right\}$, the large inequality holds true, and equality occurs precisely for indices $(k, l)=\left(M_{\rho, i}, M_{\rho, j}\right)$ with distinct \\ $i, j\in\left\{1, 2, \hdots, \left\lfloor\tfrac{1-t\rho}{2-(1+2t)\rho}\right\rfloor\right\}$.
\end{proposition}

\begin{proof}
Let $i, j\in\mathbb{N}\setminus\left\{0\right\}$. Once again, it is enough to show 

$$\underset{\substack{k\in Y_{\rho, i} \\ l\in Y_{\rho, j} \\ k\neq l}}\max \left(x_{N_{\rho, k}', N_{\rho, l}'}^\rho+y_{N_{\rho, k}', N_{\rho, l}'}^\rho\right)<2d_\rho=\tfrac{2M_{\rho, 1}}{2\left\lfloor\tfrac{1}{\rho}\right\rfloor M_{\rho, 1}+1}$$

The case when $i=j$ works exactly the same way as in the proof of Proposition~\ref{prop:x+y<2d, rho not in E}. If $i\neq j$, then in the proof of Proposition \ref{prop:x+y<2d, rho not in E} we had the condition

$$\begin{aligned}
&\underset{\substack{k\in Y_{\rho, i} \\ l\in Y_{\rho, j} \\ k\neq l}}\max \left(x_{N_{\rho, k}', N_{\rho, l}'}^\rho+y_{N_{\rho, k}', N_{\rho, l}'}^\rho\right)<2d_\rho \\
\iff& M_{\rho, i}+M_{\rho, j}<2M_{\rho, 1}(i+j-1) \\
\iff& M_{\rho, i}+M_{\rho, j}<2(i+j-1) \text{ since } M_{\rho, 1}=1 \text{ because }\rho\in E \\
\iff& M_{\rho, i}<2i-1 \text{ or }M_{\rho, j}<2j-1 \text{ according to Lemma \ref{lem:bounds on maxima, <0.5}} \\
\iff& v_{\rho, i}>0 \text{ or }v_{\rho, j}>0
\end{aligned}$$

Therefore according to Lemma \ref{lem:equality case in bounds on maxima lemma, <0.5}

\begin{itemize}
\item if $\rho\in F$, this condition is fulfilled for any distinct $i, j\in\mathbb{N}\setminus\left\{0\right\}$, since $\left\lfloor\tfrac{1-t\rho}{2-(1+2t)\rho}\right\rfloor=1$.
\item if $\rho\in G_t$ with $t\in\mathbb{N}\setminus\left\{0, 1\right\}$, then this condition is not fulfilled (meaning that we have an equality instead of strict inequality in the initial statement) precisely for distinct $i, j\in\left\{1, 2, \hdots, \left\lfloor\tfrac{1-2\rho}{3\rho-2}\right\rfloor\right\}$ since $\left\lfloor\tfrac{1-t\rho}{2-(1+2t)\rho}\right\rfloor\geqslant2$.
\end{itemize}
\end{proof} 

For any $t\in\mathbb{N}\setminus\left\{0, 1\right\}$, we will denote $j_{\rho, t} = \left\lfloor\tfrac{1-t\rho}{2-(1+2t)\rho}\right\rfloor$ in the following theorem.

\begin{theorem}[When the best diagonal point is an extreme point]\label{thm:is the best diagonal point an extreme point ?, <0.5}
 If $\rho\in G_t$ with $t\in\mathbb{N}\setminus\left\{0, 1\right\}$, then $D_\rho$ is not an extreme point of $\Lambda_\rho$; instead it is the middle point of points $A^\rho_{N_{\rho, M_{\rho, 1}}', N_{\rho, M_{\rho, j_{\rho, t}}}'}$ and $A^\rho_{N_{\rho, M_{\rho, j_{\rho, t}}}', N_{\rho, M_{\rho, 1}}'}$, which belong to line $x+y=2d_\rho$ and are extreme points of $\Lambda_\rho$.  If $\rho\in F$ or $\rho\notin E$, then $D_\rho$ is an extreme point of $\Lambda_\rho$.
\end{theorem}

\begin{proof}
\begin{itemize}
\item Suppose $\rho\in F$. Then according to Propositions \ref{prop:x+y<2d, rho not in E} and \ref{prop:x+y<2d, rho in E}, for all distinct $m, n\in I_\rho$, we have

$$x_{m, n}^\rho+y_{m, n}^\rho<2d_\rho$$

and hence $D_\rho$ is an extreme point of $\Lambda_\rho$.

\item Suppose $\rho\in G_t$ with $t\in\mathbb{N}\setminus\left\{0, 1\right\}$.

\begin{figure}[!h]
\begin{center}
\includegraphics[width=11cm]{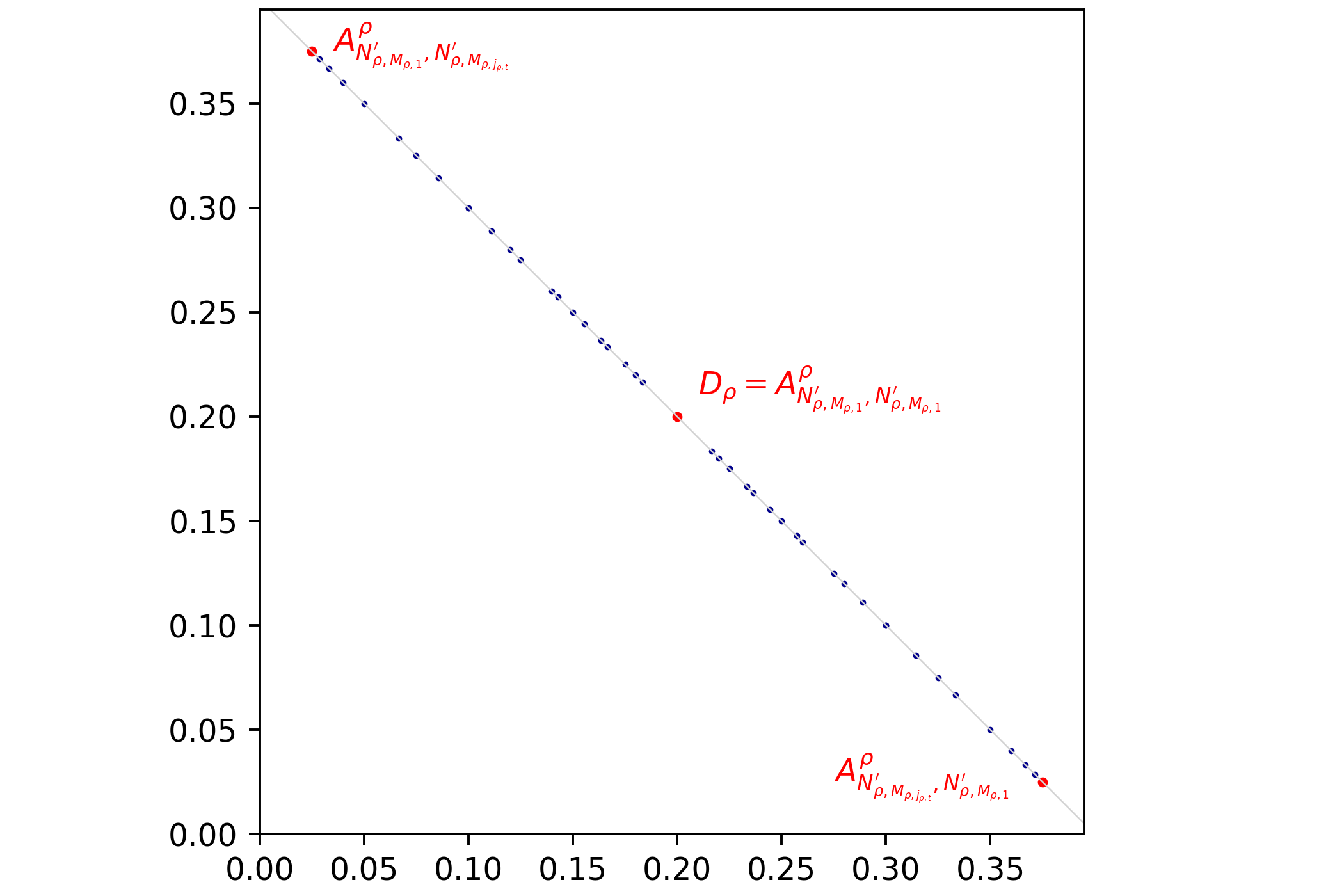}
\end{center}
    \caption{Points $A_{N_{\rho, M_{\rho, i}}', N_{\rho, M_{\rho, j}}'}^\rho$ for distinct $i, j\in\left\{1, 2, \hdots, j_{\rho, t}\right\}$ for $\rho=0.395+\pi.10^{-5}$ ($\rho\in G_t$ with $t=2$). Straight line $y=-x+2d_\rho$ is in gray.}
    %\label{fig:Fig. 7}
\end{figure}

According to the proof of Proposition \ref{prop:x+y<2d, rho in E}, we have for any distinct \\ $i, j\in\left\{1, 2, \hdots, j_{\rho, t}\right\}$

$$x_{N_{\rho, M_{\rho, i}}', N_{\rho, M_{\rho, j}}'}^\rho+y_{N_{\rho, M_{\rho, i}}', N_{\rho, M_{\rho, j}}'}^\rho = 2d_\rho$$

which means that such points $A_{N_{\rho, M_{\rho, i}}', N_{\rho, M_{\rho, j}}'}^\rho$ (and $D_\rho$) lie on line $x+y=2d_\rho$ (and those are the only points of $L_\rho$ lying on this line). Additionally, for any such indices $i, j$, we have 

$$\begin{aligned}
x_{N_{\rho, M_{\rho, i}}', N_{\rho, M_{\rho, j}}'}^\rho &= \tfrac{M_{\rho, i}}{\left(M_{\rho, i}+M_{\rho, j}\right)\left\lfloor\tfrac{1}{\rho}\right\rfloor+i+j-1} \\
&= \tfrac{2i-1}{2\left(i+j-1\right)\left\lfloor\tfrac{1}{\rho}\right\rfloor+i+j-1}
\end{aligned}$$

because $v_{\rho, i}=v_{\rho, j}=0$ (see Lemma \ref{lem:equality case in bounds on maxima lemma, <0.5}). This sequence is strictly decreasing with respect to $j$ and strictly increasing with respect to $i$. Let's prove this last claim. For any $i, j\in\mathbb{N}\setminus\left\{0\right\}$, we have

$$\begin{aligned}
&\tfrac{2i+1}{2(i+j)\left\lfloor\tfrac{1}{\rho}\right\rfloor+i+j+2}>\tfrac{2i-1}{2(i+j-1)\left\lfloor\tfrac{1}{\rho}\right\rfloor+i+j+1} \\
\iff&(2i+1)\left(2(i+j-1)\left\lfloor\tfrac{1}{\rho}\right\rfloor+i+j+1\right)>(2i-1)\left(2(i+j)\left\lfloor\tfrac{1}{\rho}\right\rfloor+i+j+2\right) \\
\iff& (4j-2)\left\lfloor\tfrac{1}{\rho}\right\rfloor+2j+3>0
\end{aligned}$$

which is true, since $j\geqslant1$. \\

Therefore $A^\rho_{N_{\rho, M_{\rho, 1}}', N_{\rho, M_{\rho, j_{\rho, t}}}'}$ and $A^\rho_{N_{\rho, M_{\rho, j_{\rho, t}}}', N_{\rho, M_{\rho, 1}}'}$ have respectively minimal and maximal $x$ coordinate among points of $L_\rho$ lying on line $y=-x+2d_\rho$. As a consequence,  $A^\rho_{N_{\rho, M_{\rho, 1}}', N_{\rho, M_{\rho, j_{\rho, t}}}'}$ and $A^\rho_{N_{\rho, M_{\rho, j_{\rho, t}}}', N_{\rho, M_{\rho, 1}}'}$ are extreme points of $\Lambda_\rho$. \\

Finally, by symmetry, the middle point of points $A^\rho_{N_{\rho, M_{\rho, 1}}', N_{\rho, M_{\rho, j_{\rho, t}}}'}, A^\rho_{N_{\rho, M_{\rho, j_{\rho, t}}}', N_{\rho, M_{\rho, 1}}'}$ is on the diagonal and therefore it is precisely $D_\rho$ (which shows that $D_\rho$ is not an extreme point of $\Lambda_\rho$).
\end{itemize}
\end{proof}

Note that for any $t\in\mathbb{N}\setminus\left\{0, 1\right\}$, if $\rho\in G_t$, 

$$A^\rho_{N_{\rho, M_{\rho, 1}}', N_{\rho, M_{\rho, j_{\rho, t}}}'}=\left(\tfrac{M_{\rho, 1}}{\left(M_{\rho, 1}+M_{\rho, j_{\rho, t}}\right)\left\lfloor\tfrac{1}{\rho}\right\rfloor+j_{\rho, t}}, \tfrac{M_{\rho, j_{\rho, t}}}{\left(M_{\rho, 1}+M_{\rho, j_{\rho, t}}\right)\left\lfloor\tfrac{1}{\rho}\right\rfloor+j_{\rho, t}}\right)$$

\section{Roundness}\label{sec:4}

In this section, we focus on the behavior of the set $\Lambda_\rho$ with respect to $\rho$. In order to determine the roundness of $\Lambda'_\rho$ we study how close the set $\Lambda_\rho$ is to a quarter of a circle with radius $\rho$. In this way, we obtain lower and upper bounds on roundness of $\Lambda'_\rho$, which show that sets $\Lambda_\rho$ (for different values of $\rho$) are not proportional "rescalings" of a single set.

Note that the upper-right quarter of a circle with radius $\rho$ and center $(0, 0)$ contains $\Lambda_\rho$, since for any $m, n\in I_\rho$, we have

$$\begin{aligned}
&\left(x_{m, n}^\rho\right)^2+\left(y_{m, n}^\rho\right)^2<\rho^2 \\
\iff&\left\lceil m\rho\right\rceil^2+\left\lceil n\rho\right\rceil^2<\rho^2\left(m+n+1\right)^2 \\
\iff&\left(m\rho+\alpha_m^\rho\right)^2+\left(n\rho+\alpha_n^\rho\right)^2<\rho^2\left(m+n+1\right)^2
\end{aligned}$$

which is true since 

$$\alpha_m^\rho<\rho \text{ and }\alpha_n^\rho<\rho$$

\subsection{Supremum of slopes from point $(0, \rho)$}

In this subsection, we would like to determine the supremum of slopes between point $(0, \rho)$ and any point of $\Lambda_\rho$. Therefore it is sufficient to study slopes between point $(0, \rho)$ and any point $A_{m, n}^\rho$ belonging to $\Lambda_\rho$. 

\begin{definition}\label{def:slope (0, rho)-Amn}
For any $m, n\in I_\rho$, denote $\gamma_{m, n}^\rho$ the slope from point $(0, \rho)$ to point $A_{m, n}^\rho$. Note that 

$$\begin{aligned}
\gamma_{m, n}^\rho&\overset{\text{def}}=\tfrac{y_{m, n}^\rho-\rho}{x_{m, n}^\rho} \\
&=-1+\tfrac{\alpha_m^\rho+\alpha_n^\rho-\rho}{\left\lceil m\rho\right\rceil}
\end{aligned}$$

\end{definition}

\begin{proposition}\label{prop:maximal slope, >0.5}
If $\rho>\tfrac{1}{2}$, then 

$$\underset{m, n\in I_\rho}\sup\gamma_{m, n}^\rho = -\rho$$
\end{proposition}

\begin{proof}
\begin{itemize}
\item We remind the reader that for any $n\in I_\rho$, we have $\alpha_n^\rho<\rho$. Moreover, according to Proposition \ref{prop:partition, >0.5}, for any $k\in\mathbb{N}$, we have 

$$\begin{aligned}
\alpha_{N_{\rho, k}}^\rho &= N_{\rho, k}-k-N_{\rho, k}\rho \\
&= \left\lfloor\tfrac{k+1}{1-\rho}\right\rfloor(1-\rho)-(k+1)+\rho \\
&= \rho-D\left(\tfrac{k+1}{1-\rho}\right)(1-\rho)
\end{aligned}$$

Therefore, since $\tfrac{1}{1-\rho}\notin\mathbb{Q}$, sequence $\left(D\left(\tfrac{k+1}{1-\rho}\right)\right)_{k\in\mathbb{N}}$ is dense in $[0, 1]$, thus we have a subsequence $(k_l)_{l\in\mathbb{N}}$ such that 

$$D\left(\tfrac{k_l+1}{1-\rho}\right)\underset{l\to\infty}\rightarrow0 \textit{ i.e. }\alpha_{N_{\rho, k_l}}^\rho\underset{l\to\infty}\rightarrow \rho$$

Hence,

$$\begin{aligned}
\underset{m, n\in I_\rho}\sup\gamma_{m, n}^\rho &= -1+\underset{m\in I_\rho}\sup\underset{n\in I_\rho}\sup\tfrac{\alpha_m^\rho+\alpha_n^\rho-\rho}{\lceil m\rho\rceil} \\
&=-1+\underset{m\in I_\rho}\sup\tfrac{\alpha_m^\rho}{\lceil m\rho\rceil}
\end{aligned}$$

\item Additionally, for any $k\in\mathbb{N}$ and $m\in I_{\rho, k}$, we have

$$\begin{aligned}
\tfrac{\alpha_m^\rho}{\lceil m\rho\rceil} &= \tfrac{m-k-m\rho}{m-k} \\
&= 1-\rho\tfrac{m}{m-k}
\end{aligned}$$

thus sequence $\left(\tfrac{\alpha_m^\rho}{\lceil m\rho\rceil}\right)_{m\in I_\rho}$ is increasing on each interval $I_{\rho, k}$ where $k\in\mathbb{N}$. Hence 

$$\underset{m\in I_\rho}\sup\tfrac{\alpha_m^\rho}{\lceil m\rho\rceil} = \underset{k\in\mathbb{N}}\sup\tfrac{\alpha_{N_{\rho, k}}^\rho}{\lceil N_{\rho, k}\rho\rceil}$$

Then for any $k\in\mathbb{N}\setminus\left\{0\right\}$, we have

$$\tfrac{\alpha^\rho_{N_{\rho, k}}}{\lceil N_{\rho, k}\rho\rceil}<\tfrac{\alpha^\rho_{N_{\rho, 0}}}{\lceil N_{\rho, 0}\rho\rceil}\iff\tfrac{N_{\rho, k}}{N_{\rho, k}-k}>1$$

which is true. Therefore

$$\begin{aligned}
\underset{k\in\mathbb{N}}\sup\tfrac{\alpha_{N_{\rho, k}}^\rho}{\lceil N_{\rho, k}\rho\rceil} &= \tfrac{\alpha_{N_{\rho, 0}}^\rho}{\lceil N_{\rho, 0}\rho\rceil} \\ 
&= 1-\rho
\end{aligned}$$

and finally

$$\underset{m, n\in I_\rho}\sup\gamma_{m, n}^\rho = -\rho$$

\end{itemize}
\end{proof}

\begin{proposition}\label{prop:maximal slope, <0.5}
If $\rho<\tfrac{1}{2}$, then 

$$\underset{m, n\in I_\rho}\sup\gamma_{m, n}^\rho = -\rho\left\lfloor\tfrac{1}{\rho}\right\rfloor$$
\end{proposition}

\begin{proof}
\begin{itemize}
\item We remind the reader that according to Proposition \ref{prop:partition, <0.5}

$$I_\rho = \left\{\left\lfloor\tfrac{k}{\rho}\right\rfloor\,|\,k\in\mathbb{N}\setminus\left\{0\right\}\right\}$$

and for any $k\in\mathbb{N}\setminus\left\{0\right\}$, we have 

$$\begin{aligned}
\alpha_{\left\lfloor\tfrac{k}{\rho}\right\rfloor}^\rho &= k-\left\lfloor\tfrac{k}{\rho}\right\rfloor\rho \\
&= \rho D\left(\tfrac{k}{\rho}\right) \\
\end{aligned}$$

Therefore, since $\tfrac{1}{\rho}\notin\mathbb{Q}$, sequence $\left(D\left(\tfrac{k}{\rho}\right)\right)_{k\in\mathbb{N}\setminus\left\{0\right\}}$ is dense in $[0, 1]$, thus we have a subsequence $(k_l)_{l\in\mathbb{N}}$ such that 

$$D\left(\tfrac{k_l}{\rho}\right)\underset{l\to\infty}\rightarrow1 \textit{ i.e. }\alpha_{\left\lfloor\tfrac{k_l}{\rho}\right\rfloor}^\rho\underset{l\to\infty}\rightarrow \rho$$

Hence,

$$\begin{aligned}
\underset{m, n\in I_\rho}\sup\gamma_{m, n}^\rho &= -1+\underset{m\in I_\rho}\sup\underset{n\in I_\rho}\sup\tfrac{\alpha_m^\rho+\alpha_n^\rho-\rho}{\lceil m\rho\rceil} \\
&=-1+\underset{m\in I_\rho}\sup\tfrac{\alpha_m^\rho}{\lceil m\rho\rceil} \\
&=-\underset{k\in\mathbb{N}\setminus\left\{0\right\}}\inf\tfrac{\rho}{k}\left\lfloor\tfrac{k}{\rho}\right\rfloor
\end{aligned}$$

\item Additionally, for any $k\in\mathbb{N}\setminus\left\{0\right\}$, we have 

$$\begin{aligned}
\left\lfloor\tfrac{k}{\rho}\right\rfloor&>\tfrac{k}{\rho}-1 \\
&>k\left\lfloor\tfrac{1}{\rho}\right\rfloor-1 \\
& \\
\textit{i.e. }\left\lfloor\tfrac{k}{\rho}\right\rfloor&\geqslant k\left\lfloor\tfrac{1}{\rho}\right\rfloor \\
\text{hence }\tfrac{\rho}{k}\left\lfloor\tfrac{k}{\rho}\right\rfloor&\geqslant\rho\left\lfloor\tfrac{1}{\rho}\right\rfloor
\end{aligned}$$

Therefore

$$\underset{m, n\in I_\rho}\sup\gamma_{m, n}^\rho = -\rho\left\lfloor\tfrac{1}{\rho}\right\rfloor$$
\end{itemize}

\end{proof}

\begin{definition}\label{def:best slope}
In the following,  $\gamma_\rho$ deontes the supremum of slopes obtained in Propositions \ref{prop:maximal slope, >0.5} and \ref{prop:maximal slope, <0.5}.
\end{definition}

Note that according to Propositions \ref{prop:maximal slope, >0.5} and \ref{prop:maximal slope, <0.5}, the following formula is true for any $\rho\in(0, 1)\cap\left(\mathbb{R}\setminus\mathbb{Q}\right)$ :

$$\gamma_\rho = -\rho\left\lfloor\tfrac{1}{\rho}\right\rfloor$$

\subsection{Bounds on $R_\rho$, the roundness of $\Lambda'_\rho$}

Thanks to the supremum of slopes from the point $(0, \rho)$ determined in the previous subsection, we obtain the roundness bounds of $\Lambda_\rho$ stated in Theorem \ref{thm:bounds on roundness}. 
\begin{figure}[!h]
\begin{center}
\includegraphics[width=11cm]{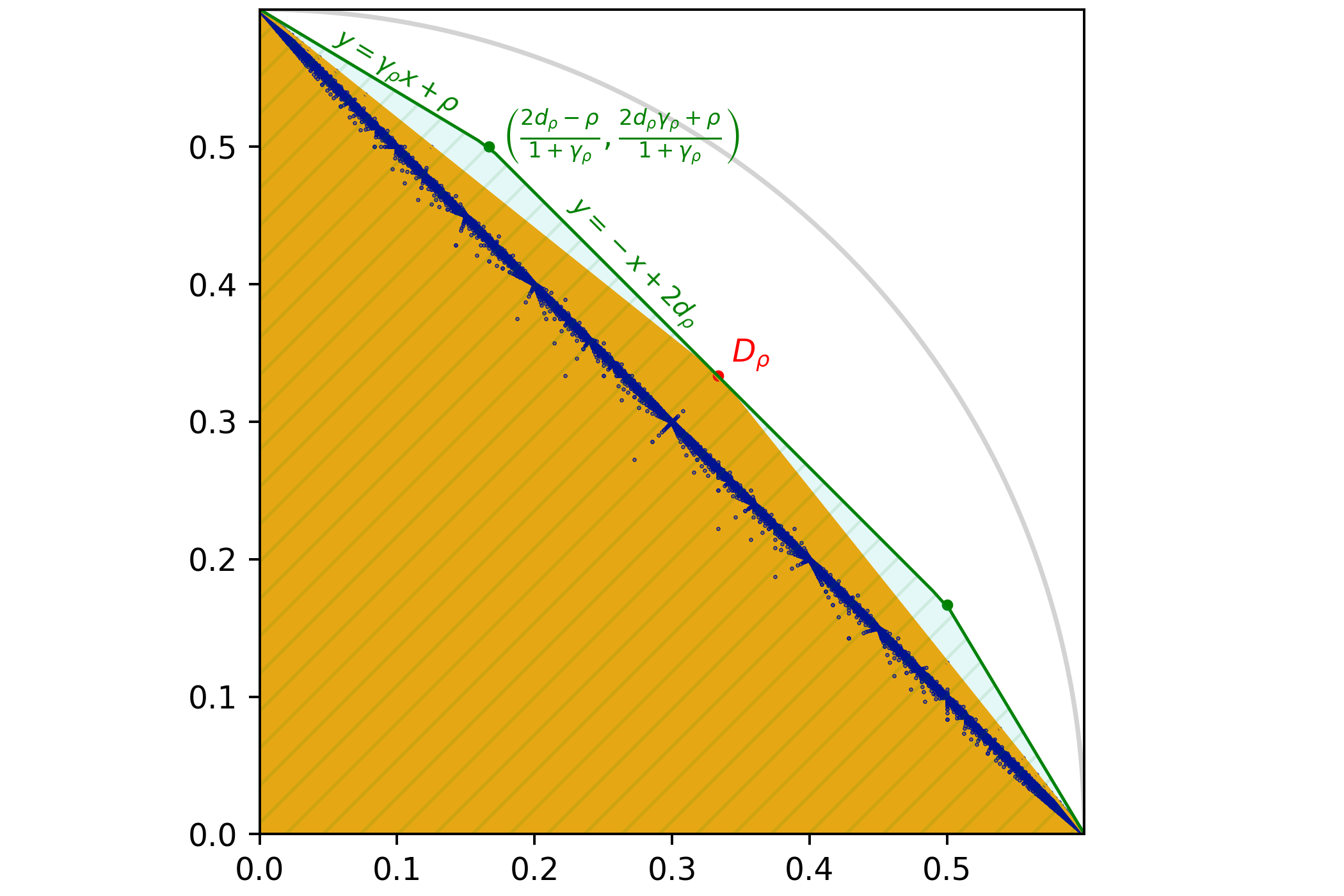}
\end{center}
    \caption{$\rho=0.6-\pi.10^{-5}$. The orange area is defined by the polygon with vertices $(0, 0), (0, \rho), (\rho, 0), D_\rho$, where $D_\rho=(d_\rho,d_\rho)$. The hatched area is defined by the polygon with vertices $(0, 0), (0, \rho), (\rho, 0), \left(\tfrac{2d_\rho-\rho}{1+\gamma_\rho}, \tfrac{2d_\rho\gamma_\rho+\rho}{1+\gamma_\rho}\right), \left(\tfrac{2d_\rho\gamma_\rho+\rho}{1+\gamma_\rho}, \tfrac{2d_\rho-\rho}{1+\gamma_\rho}\right)$. The hatched area contains $\Lambda_\rho$, which contains the orange area.}
    %\label{fig:Fig. 7}
\end{figure}
\begin{proof}(of Theorem \ref{thm:bounds on roundness})

\begin{itemize}

\item According to Theorems \ref{thm:best diagonal point, >0.5} and \ref{thm:best diagonal point, <0.5}, we know that the polygon with vertices $(0, 0), (0, \rho), (\rho, 0), D_\rho$ is contained in $\Lambda_\rho$. Additionally, its area is $\rho d_\rho$, hence 

$$R_\rho\geqslant \tfrac{4}{\pi}\tfrac{d_\rho}{\rho}$$

\item Note that the intersection point of straight lines $y=-x+2d_\rho$ and $y=\gamma_\rho x+\rho$ is $\left(\tfrac{2d_\rho-\rho}{1+\gamma_\rho}, \tfrac{2d_\rho\gamma_\rho+\rho}{1+\gamma_\rho}\right)$. Thanks to Propositions \ref{prop:maximal slope, >0.5} and \ref{prop:maximal slope, <0.5} (and since $\Lambda_\rho$ is symmetrical with respect to the diagonal), we know that $\Lambda_\rho$ is contained in polygon with vertices $(0, 0), (0, \rho), (\rho, 0), \left(\tfrac{2d_\rho-\rho}{1+\gamma_\rho}, \tfrac{2d_\rho\gamma_\rho+\rho}{1+\gamma_\rho}\right), \left(\tfrac{2d_\rho\gamma_\rho+\rho}{1+\gamma_\rho}, \tfrac{2d_\rho-\rho}{1+\gamma_\rho}\right)$. By symmetry, it's area is 

$$\begin{aligned}
&2\left(\tfrac{1}{2}\rho\tfrac{2d_\rho-\rho}{1+\gamma_\rho}+\tfrac{1}{2}\sqrt{2}d_\rho\sqrt{\left(d_\rho-\tfrac{2d_\rho-\rho}{1+\gamma_\rho}\right)^2+\left(d_\rho-\tfrac{2d_\rho\gamma_\rho+\rho}{1+\gamma_\rho}\right)^2}\right) \\
=&\,\rho\tfrac{2d_\rho-\rho}{1+\gamma_\rho}+\tfrac{2d_\rho}{1+\gamma_\rho}\left(d_\rho\left(\gamma_\rho-1\right)+\rho\right) \\
=&\,\tfrac{1}{1+\gamma_\rho}\left(-\rho^2+4\rho d_\rho-2\left(1-\gamma_\rho\right)d_\rho^2\right)
\end{aligned}$$

%\begin{figure}[!h]
%\begin{center}
%\includegraphics[width=11cm]{plot_proof_thm6.png}
%\end{center}
 %   \caption{Points $A_{m, n}^\rho$ (dark blue) such that $m, n\in I_\rho\cap\llbracket1, 1000\llbracket$ for $\rho=0.6-\pi.10^{-5}$. Orange area is defined by the polygon with vertices $(0, 0), (0, \rho), (\rho, 0), D_\rho$. Hatched area is defined by the polygon with vertices $(0, 0), (0, \rho), (\rho, 0), \left(\tfrac{2d_\rho-\rho}{1+\gamma_\rho}, \tfrac{2d_\rho\gamma_\rho+\rho}{1+\gamma_\rho}\right), \left(\tfrac{2d_\rho\gamma_\rho+\rho}{1+\gamma_\rho}, \tfrac{2d_\rho-\rho}{1+\gamma_\rho}\right)$. Hatched area contains $\Lambda_\rho$, which contains orange area.}
    %\label{fig:Fig. 7}
%\end{figure}

Note that at the second line, we used the fact that $d_\rho\left(\gamma_\rho-1\right)+\rho\geqslant0$. Indeed, we have 

$$\begin{aligned}
&d_\rho\left(\gamma_\rho-1\right)+\rho\geqslant0 \\
\iff&\gamma_\rho\geqslant\tfrac{d_\rho-\rho}{d_\rho}
\end{aligned}$$

which is true by definition of $\gamma_\rho$ since $\tfrac{d_\rho-\rho}{d_\rho}$ is the slope from point $(0, \rho)$ to point $D_\rho$. Therefore, since $\gamma_\rho=-\rho\left\lfloor\tfrac{1}{\rho}\right\rfloor$, we obtain the following upper bound 

$$R_\rho\leqslant\tfrac{4}{\pi}\tfrac{1}{1-\rho\left\lfloor\tfrac{1}{\rho}\right\rfloor}\left(-1+4\tfrac{d_\rho}{\rho}-2\left(1+\rho\left\lfloor\tfrac{1}{\rho}\right\rfloor\right)\tfrac{d_\rho^2}{\rho^2}\right)$$

\end{itemize}
\end{proof}

Since the function $\rho\mapsto d_\rho$ has jumps around rational values $\tfrac{1}{k}$ and $1-\tfrac{1}{k}$ for any $k\in\mathbb{N}\setminus\left\{0, 1\right\}$, we conclude with the following

\begin{conjecture}[Jumps of roundness function]\label{conj:jumps of roundness function}
The roundness function $\rho\mapsto R_\rho\left(\Lambda_\rho\right)$ has jumps around $\tfrac{1}{k}$ and $1-\tfrac{1}{k}$ for any $k\in\mathbb{N}\setminus\left\{0, 1\right\}$.
\end{conjecture}

\section{Conclusions}

We determined the expression of the diagonal point of $L_\rho$ with maximal coordinates, namely the best diagonal point (Theorems \ref{thm:best diagonal point, >0.5} and \ref{thm:best diagonal point, <0.5}). We proved that either the best diagonal point is an extreme point of $\Lambda_\rho$, or $\Lambda_\rho$ has no extreme point on the diagonal and has an edge whose middle point is the best diagonal point, and that both cases occur (Theorems \ref{thm:is the best diagonal point an extreme point ?, >0.5} and \ref{thm:is the best diagonal point an extreme point ?, <0.5}). This shows that the picture given by Kwapisz in \cite{Kwapisz2} is misleading (on that picture, $\Lambda_\rho$ does not have a diagonal vertex). We determined lower and upper bounds for the roundness of $\Lambda'_\rho$ (Theorem \ref{thm:bounds on roundness}), which confirms that the idea that sets $\Lambda_\rho$ would be "rescalings" of one single set is wrong. This is particularly underlined by Conjecture \ref{conj:jumps of roundness function} and Figure \ref{fig:roundness}. \\

\section{Appendix}
In this section we sketch how to adjust Kwapisz's construction to get $C^1$-diffeomorphisms with rotation sets that lie in all four quadrants of the plane. 
\begin{proof}(of Theorem \ref{four})
First, we observe the following.
\begin{claim}
\[
\begin{array}{l}
\text{Conv}\left(\left\{(0, \pm\rho), (\pm\rho, 0)\right\}\cup\left\{\left(\tfrac{\pm\lceil m\rho\rceil}{m+n+1}, \tfrac{\pm\lceil n\rho\rceil}{m+n+1}\right)\,|\,m, n\in\mathbb{N}_0, \alpha_m^\rho, \alpha_n^\rho<\rho\right\}\right) = \\
\text{Conv}\left(\left\{(0, \pm\rho), (\pm\rho, 0)\right\}\cup\left\{\left(\tfrac{\lceil m\rho\rceil-\lceil m'\rho\rceil}{m+m'+n+n'+1}, \tfrac{\lceil n\rho\rceil-\lceil n'\rho\rceil}{m+m'+n+n'+1}\right)\,|\,m,m',n,n'\in\mathbb{N}_0, \alpha_m^\rho, \alpha_{m'}^\rho,\alpha_n^\rho,\alpha_{n'}^\rho<\rho\right
\}\right)  \end{array}
\]

\end{claim}

We modify Kwapisz's construction by starting with the one-dimensional model. First, we add two extra circles $\mathbb{S}^{(h')}$ and $\mathbb{S}^{(v')}$ to the two $\mathbb{S}^{(h)},\mathbb{S}^{(v)}$ used by Kwapisz. Our phase space $X$ is now a bouquet of four circles. Let $\psi$ be a Denjoy $C^1$-diffeomorphism of $\mathbb{S}^1=\mathbb{R}/\mathbb{Z}$ with rotation number $\rho$, and a wandering interval $J$ symmetric about the origin, where we identify subsets of $\mathbb{S}^1$ with a lift to $\mathbb{R}$ for simplicity. Let $p:\mathbb{S}^1\to\mathbb{S}^1$ be a map that collapses $I=\frac{17}{18}\cdot J$ to zero and is affine elsewhere, so that the monotone map $\phi:=\psi\circ p$ with a single flat plateau is smooth. Let $\eta$ be the map $[-1,1] \to [0,1]$  given by the formula $\eta (x) := 1-x^2$. We let $\tau$ denote the $\psi$-image of the origin and let $\eta_\tau:I\to[0,\tau]$ be the map of subsets of $S^1$ whose lift is a linear rescaling of $\eta$ to a transformation from $I$ to $[0, \tau]$. We define $-\eta_\tau:I \to  [-\tau,0]$ similarly. The following definition gives for each $\sigma\in  \{h,h',v,v'\}$ a continuous map $F^{(\sigma)} :X\to X$. We have the `projection' $X\to \mathbb{S}^1$ sending $x\to \underline{x}$, with the relation $\underline{x}^{(\sigma)}=x$ for $x\in \mathbb{S}^{\sigma} \subset X$. We set
\[
   \pm F^{(\sigma)}(x)=
\begin{cases}
    \psi^{\pm1}(\underline{x})^{(\sigma)},& \text{if } x\in \mathbb{S}^{(\sigma)}\\
  \pm  \eta_\tau(\underline{x})^{(\sigma)},& \text{if } x\in I^{(\tilde{\sigma})}, \tilde{\sigma}\neq\sigma\\
   p(\underline{x})^{(\tilde{\sigma})}& \text{if } x\in \mathbb{S}^{(\tilde{\sigma})}\setminus I^{(\tilde{\sigma})}, \tilde{\sigma} \neq\sigma.\\
\end{cases}
\]

Now we define the map $F$ as follows
$$F:=(-F^{(h')})\circ (-F^{(v')})\circ F^{(h)}\circ F^{(v)}.$$

Now we embed the bouquet $X$ into $\mathbb{T}^2$, with the circles $\mathbb{S}^{(h')}$ and $\mathbb{S}^{(h)}$ almost horizontal once essential circles, and the circles $\mathbb{S}^{(v')}$ and $\mathbb{S}^{(v)}$ almost vertical once essential circles. For example, considering the usual universal covering map $\pi:\mathbb{R}^2\to\mathbb{T}^2$ one can let $X$ in $\mathbb{T}^2$ be  
$$S^{(h)}=\pi(\{(z_1,z_2)\in [0,1]^2:z_2=\frac{1}{10}\sin^2(\pi z_1)\}),$$ 
$$S^{(h')}=\pi(\{(z_1,z_2)\in [0,1]^2:1-z_2=\frac{1}{10}\sin^2(\pi z_1)\}),$$ 
$$S^{(v)}=\pi(\{(z_1,z_2)\in [0,1]^2:z_1=\frac{1}{10}\sin^2(\pi z_2)\}),$$ 
$$S^{(v')}=\pi(\{(z_1,z_2)\in [0,1]^2:1-z_1=\frac{1}{10}\sin^2(\pi z_2)\}).$$ Let $\tilde{X}=\pi^{-1}(X)$. Since all the maps defined so far are homotopic to the identity, one can obtain lifts $\pm\tilde{F}^{(\sigma)}$ and
$$\tilde{F}:=(-\tilde{F}^{(h')})\circ (-\tilde{F}^{(v')})\circ \tilde{F}^{(h)}\circ \tilde{F}^{(v)}$$
\begin{figure}
    \centering
    \includegraphics[width=0.5\linewidth]{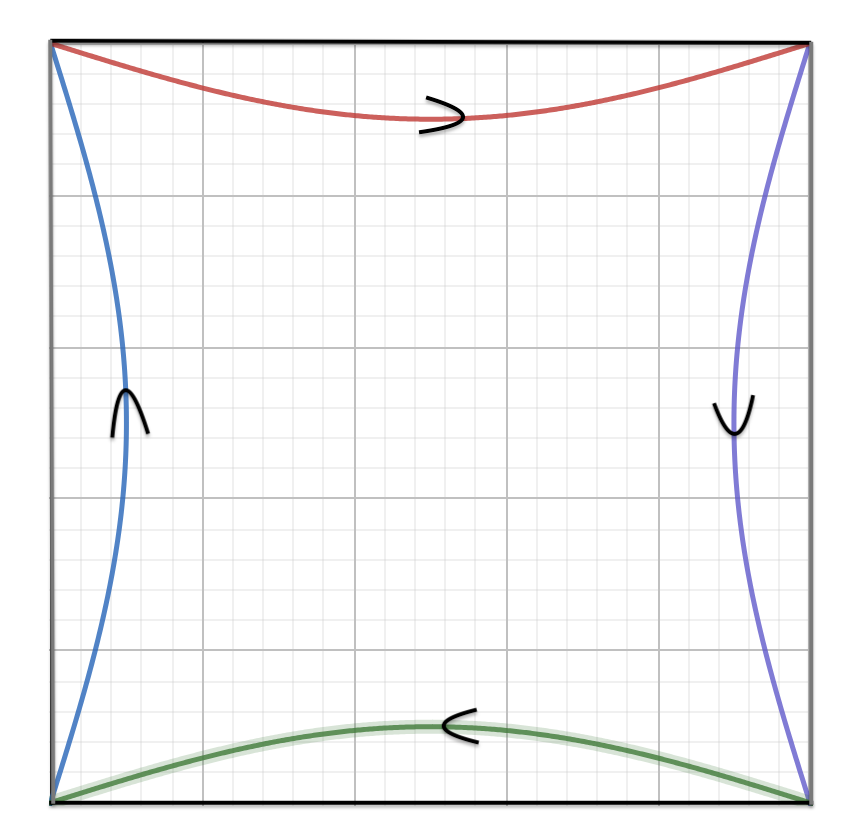}
    \caption{The embeding of $X$ in $\mathbb{T}^2$}
    \label{fig:enter-label}
\end{figure}
The following propositions are proven analogously to those in \cite{Kwapisz2}.
\begin{proposition}
\begin{itemize}
    \item  If for $\tilde x\in\tilde{X}$ the limit $\rho (\tilde{F},\tilde{x}) =\lim_{n\to\infty} \frac{1}{n}(\tilde{F}^n(\tilde x)-\tilde x)$ exists, then
$\rho(\tilde{F},\tilde{x})\in \Lambda_\rho.$
\item For any $m,n\in\mathbb{N}$ with $\alpha_m,\alpha_n<\rho$, there is $\tilde x\in \tilde X$ with $\rho(\tilde{F},\tilde{x})=\left(\tfrac{\lceil \pm m\rho\rceil}{m+n+1}, \tfrac{\lceil \pm n\rho\rceil}{m+n+1}\right)$.
\end{itemize}
\end{proposition}
\begin{proposition}
$\rho(\tilde{F})=\Lambda_\rho$
\end{proposition}
%\begin{figure}
 %   \centering
  %  \includegraphics[width=0.5\linewidth]{Untitled.pdf}
   % \caption{Producing the $C^1$-embedding.}
    %\label{fig:enter-label}
%\end{figure}

Finally, we construct in stages a $C^1$ diffeomorphism $G$ of $\mathbb{T}^2$ that reflects the dynamics of $F$. First we define a diffeomorphism corresponding to each $\pm F^{(\sigma)}$. The construction is almost identical to that of Kwapisz, and so we only highlight the differences. First form a neighbourhood $U$ of $X$ in $\mathbb{T}^2$ consisting of thin strips about each $X^{(\rho)}$ which intersect only in a symmetric neighbourhood of the origin, the common point of all the $X^{(\rho)}$. The complement $\mathbb{T}^2 \setminus U$ is then $3$ topological disks $D_i$, $i=1,2,3$.  We then form a retraction $r\colon U\setminus C \to X\setminus C$ outside of a symmetric neighbourhood $C$ of the origin so that the pre-images of points under $r$ are small arcs transverse to $X$. The neighbourhood $C$ needs to be chosen sufficiently small so that it corresponds to intervals in the $X^{(\rho)}$ with image within the wandering interval as detailed by Kwapisz. The $C^1$ diffeomorphism, denoted $\pm G^{(\sigma)}$, is then defined on the closure of $U$ in two stages. First, contract the closure of $U$ so that the arcs forming the pre-images of $r$ are preserved under the contraction. Then, we map the contracted image of $C$ diffeomorphically to its image so that the $8$ strips emanating from the origin do not intersect and stay within the pre-image under $r$ of the wandering interval in $X^{(\sigma)}$. Then, outside the closure of $C$, identifying the arcs in the pre-images of $r$ with their images under $r$, map the arcs as indicated by the map $\pm F^{(\sigma)}$, where the arcs are mapped to their image arc diffeomorphically and do not intersect the image of any other arcs. As the retraction $r$ semiconjugates the action of $\pm G^{(\sigma)}$ on $\cap_{n\in \mathbb{Z}^+}(\pm G^{(\sigma)})^n(U)$ away from the points corresponding to the wandering interval and its images with the action of $\pm F ^{(\sigma)}$, we have that the rotational behaviour of all points in $U$ is captured by how $\pm F^{(\sigma)}$ acts on $X$. The differences between $(\pm G^{(\sigma)})$ and $(\pm F^{(\sigma)})$ on  $C$ and its images are washed out as the wandering interval does not alter the rotation set. Then extend $\pm G^{(\sigma)}$ to all of $\mathbb{T}^2$ by introducing  a source in each of the disks $D_i$ (using the same point as a source for the various maps). Thus, with $V$ denoting the complement of the $3$ sources, we have $\cap_{n\in \mathbb{Z}^+}(\pm G^{(\sigma)})^n(V)=\cap_{n\in \mathbb{Z}^+}(\pm G^{(\sigma)})^n(U)$. 

Then we form the desired $C^1$ diffeomorphism of $\mathbb{T}^2$ with rotation set $\Lambda_\rho$ as the composition,
\[
G:=(-G^{(h')})\circ (-G^{(v')})\circ G^{(h)}\circ G^{(v)}.
\]

As the diffeomorphism $G$ constructed in this way respects the fibers of the retraction, there is no direct modification to obtain any smoothness beyond $C^1$. If one is only concerned with constructing a homeomorphism of $\mathbb{T}^2$ with rotation set $\Lambda_\rho$, then one can apply the results of Brown \cite{Brown} (see also \cite{BCH1}) to a near homeomorphism $H$ that has 3 repelling fixed points and has $X$ as an invariant attractor on which $H$ acts as $F$. This construction is significantly simpler, but it does not yield control of the smoothness.
\end{proof}

\section{Acknowledgments}
The authors thank Phil Boyland, Sylvain Crovisier and Tamara Kucherenko for useful discussions. The first author particularly thanks the 'Foundation Math\'ematique Jacques Hadamard' (FMJH) for financial support, Fr\'ed\'eric Pascal (ENS Paris-Saclay) and the second author for supervising his research internship at Jagiellonian University (Krak\'ow) which led to his contribution to this article, and Filip Wierzbowski for fruitful discussions. The second author's research was partially supported by EU funds under the project No.~CZ.02.01.01/00/23\_021/0008759, through the Operational Programme Johannes Amos Comenius.

\end{document}